\documentclass[bj]{imsart}

\usepackage[ansinew]{inputenc}
\usepackage{ upgreek }
\usepackage{graphicx} 
\usepackage{textcomp} 
\usepackage{amsfonts} 
\usepackage{amsmath}
\usepackage{amssymb}
\usepackage{mathtools}
\usepackage{latexsym}
\usepackage{array}
\usepackage{float}
\usepackage{amsthm} 
\usepackage{calc}
\usepackage{perpage}
\usepackage{textcomp}
\usepackage{verbatim} 
\usepackage{marvosym}
\usepackage{hieroglf}
\usepackage{multirow} 
\usepackage{rotating} 
\usepackage{dsfont} 

\usepackage[english]{babel}
\usepackage{hyperref} 
\usepackage{color}
\usepackage{tikz}
\usetikzlibrary{shapes,arrows,positioning,calc}

\newcommand{\normsup}[1]{\ensuremath{\!|\!| #1 | \! |_{\infty}}}
\newcommand{\normone}[1]{\ensuremath{\!|\!| #1 | \! |_{1}}}
\newcommand{\normtwo}[1]{\ensuremath{\!|\!| #1 | \! |_{2}}}

\newcommand{\normqj}[1]{\ensuremath{\!|\!| #1 | \! |_{q_j}}}
\newcommand{\normqjd}[1]{\ensuremath{\!|\!| #1 | \! |_{q_j}^*}}
\newcommand{\norm}[1]{\ensuremath{\!|\!| #1 | \! |}}

\newcommand{\normpj}[1]{\ensuremath{\!|\!| #1 | \! |_{p_j}}}

\newcommand{\inprod}[2]{\ensuremath{\langle #1 , \, #2 \rangle}}

\newtheorem{theorem}{Theorem}[section]

\newtheorem{lemma}{Lemma}[section]

\newcommand{\1}{{\rm 1}\mskip -4,5mu{\rm l} }

\newcommand{\argmin}{\mathop{\mathrm{arg\,min}}}

\def\R{\mathbb{R}}

\newcommand{\rp}{\mathbb{R}^p}
\newcommand{\rn}{\mathbb{R}^n}

\newcommand{\est}{\widehat\beta}

\newcommand{\betatrue}{\beta^*}
\newcommand{\lo}{\overline{\lambda}}

\newcommand{\genbeta}{\textcolor{black}{\beta}}
\newcommand{\genconstant}{\textcolor{black}{\ensuremath{u}}}
\newcommand{\riskconst}{\textcolor{black}{\ensuremath{a}}}
\newcommand{\Riskconst}{{\ensuremath{\bf \riskconst}}}
\newcommand{\Loss}{\textcolor{black}{\ensuremath{L}}}
\newcommand{\Lossa}{\textcolor{black}{\ensuremath{L_\Riskconst}}}
\newcommand{\constv}{\textcolor{black}{v}}
\newcommand{\classv}{\ensuremath{\mathcal F_{\constv}}}

\usepackage{pdfsync}

\RequirePackage[OT1]{fontenc}
\RequirePackage[numbers,sort]{natbib}
\RequirePackage[colorlinks,citecolor=blue,urlcolor=blue]{hyperref}
\RequirePackage{hypernat}
\setattribute{journal}{name}{Bernoulli}

\newif\ifgray
\graytrue
\grayfalse
\IfFileExists{BackgroundColor.tex}{\grayfalse
}

\ifgray
\pagecolor{black!30!white}
\fi

\begin{document}
\begin{frontmatter}
\title{Oracle Inequalities for High-dimensional~Prediction}
\runtitle{Oracle Inequalities for Prediction}

\begin{aug}
\author{\fnms{Johannes Lederer${}^{1}$}\ead[label=e1]{ledererj@uw.edu}},
\author{\fnms{Lu Yu${}^2$}\ead[label=e2]{luyu92@uw.edu}},\and
\author{\fnms{Irina Gaynanova${}^3$}\ead[label=e3]{irinag@stat.tamu.edu}}


\address[a]{Department of Statistics\\
University of Washington\\
Box 354322\\
Seattle, WA 98195\\
\printead{e1}}
\address[b]{Department of Statistics\\
University of Washington\\
Box 354322\\
Seattle, WA 98195\\
\printead{e2}}
\address[c]{Department of Statistics\\
Texas A\&M University\\	
3143 TAMU\\	
College Station, TX 77843\\
\printead{e3}}

\runauthor{Lederer, Yu \& Gaynanova}

\affiliation{${}^1$Department of Statistics at the University of Washington\\${}^2$Department of Biostatistics at the University of Washington\\${}^3$Department of Statistics at Texas A\&M University}

\end{aug}

\begin{abstract}~
  The abundance of high-dimensional data in the modern sciences has generated tremendous interest in penalized estimators such as the lasso, scaled lasso, square-root lasso, elastic net, and many others.
  In this paper, we establish a general oracle inequality for prediction in high-dimensional linear regression with such methods.
  Since the proof relies only on convexity and continuity arguments,  the result holds irrespective of the design matrix and applies to a wide range of penalized estimators.
  Overall, the bound demonstrates that generic estimators can provide consistent prediction with any design matrix.
  From a practical point of view, the bound can help to identify the potential of specific estimators, and they can help to get a sense of the prediction accuracy in a given application.
\end{abstract}


\end{frontmatter}

 \newcommand{\outcome}{\textcolor{black}{\ensuremath{Y}}}
\newcommand{\design}{\textcolor{black}{\ensuremath{X}}}
\newcommand{\noise}{\textcolor{black}{\ensuremath{\varepsilon}}}
\newcommand{\truth}{\textcolor{black}{\ensuremath{\beta^*}}}

\newcommand{\tuningparameter}{\ensuremath{\lambda}}
\newcommand{\estimator}{\ensuremath{\widehat \beta}}
\newcommand{\estimatort}{\ensuremath{\estimator^\tuningparameter}}
\newcommand{\tildeestimator}{\ensuremath{\tilde\beta^\tuningparameter}}
\newcommand{\estimatortprime}{\ensuremath{\estimator^{\tuningparameter'}}}

\newcommand{\linkfunction}{\ensuremath{g}}

\newcommand{\numbergroups}{\ensuremath{k}}

\newcommand{\normj}[1]{\textcolor{black}{\ensuremath{\norm{#1}_j}}}
\newcommand{\normjd}[1]{\textcolor{black}{\ensuremath{\norm{#1}_j^*}}}
\newcommand{\normjj}[1]{\textcolor{black}{\ensuremath{\norm{#1}_j^{j-1}}}}
\newcommand{\normjjj}[1]{\textcolor{black}{\ensuremath{\norm{#1}_j^j}}}

\newcommand{\groupset}{\ensuremath{A}}
\newcommand{\groupsets}{\ensuremath{A_1,\dots,A_\numbergroups}}
\newcommand{\groupsize}{\ensuremath{a}}
\newcommand{\groupsizes}{\ensuremath{a_1,\dots,a_\numbergroups}}
\newcommand{\fusedmatrix}{\ensuremath{M}}
\newcommand{\fusedmatrixj}{\ensuremath{M_j}}

\newcommand{\estimatororacle}{\overline \beta}

\newcommand{\projectionmatrix}{P}
\newcommand{\projectionmatrixj}{\textcolor{black}{P_j}}
\newcommand{\projectionmatrixs}{\ensuremath{\textcolor{black}{P_1,\dots,P_\numbergroups}}}

\section{Introduction} Oracle inequalities are the standard theoretical framework for measuring the accuracy of high-dimensional estimators~\cite{Buhlmann11}. Two main benefits of oracle inequalities are that they hold for finite sample sizes and that they adapt to the underlying model parameters. Oracle inequalities are thus used, for example, to compare estimators and to obtain an idea of the sample size needed in a specific application.

\newcommand{\sbo}{sparsity bound}
\newcommand{\pbo}{penalty bound}
\newcommand{\Pbo}{Penalty bound}
\newcommand{\sbos}{sparsity bounds}
\newcommand{\Sbos}{Sparsity bounds}
\newcommand{\pbos}{penalty bounds}

For high-dimensional prediction, there are two types of oracle inequalities:  so-called fast rate bounds and so-called slow rate bounds. Fast rate bounds hold for near orthogonal designs and bound the prediction error in terms of the sparsity of the regression vectors. Such bounds have been derived for a number of methods, including the lasso, the square-root lasso, and their extensions to grouped variables, see~\cite{Belloni11,Buhlmann11,Bunea_Tsybakov_Wegkamp07,Hastie15,Yoyo13,vdGeer11} and others. Slow rate bounds, on the other hand, hold for any design and bound the prediction error in terms of the penalty value of the regression vectors. Some examples of such bounds have been developed~\cite{MM11,Kolt10,RigTsy11}, but a general theory has not been established. Importantly, unlike the unfortunate naming suggests, slow rate bounds are of great interest. In particular, slow rate bounds are not inferior to fast rate bounds,  quite in contrast~\cite{ArnakMoYo14,Hebiri13}:  (i)~Slow rate bounds hold for any design, while fast rate bounds impose strong and in practice unverifiable assumptions on the correlations in the design. (ii)~Even if the assumptions hold,  fast rate bounds can contain unfavorable factors, while the factors in slow rate bounds are small, global constants. (iii)~Also in terms of rates, slow rate bounds can outmatch even the most favorable fast rate bounds. See~\cite{ArnakMoYo14} and references therein for a detailed comparison of the two types of bounds. To avoid confusion in the following, we will use the terms \emph{\pbos}  instead of slow rate bounds  and \emph{\sbos} instead of fast rate bounds.

In this paper, we develop a general \pbo\ for prediction in high-dimensional linear regression. 
This oracle inequality holds for any sample size, design, and noise distribution, and it applies to a very general family of estimators.
For established estimators such as the lasso and the square-root lasso, the result does not imply new guarantees, but it unifies existing bounds in a concise fashion.
In general, the result is a convenient resource for \pbos, and it demonstrates that prediction guarantees hold broadly in high-dimensional regression.

The organization of the paper is as follows. Below, we introduce the setting and notation and establish relationships to existing work. In Section~\ref{sec:general}, we state the general result. In Section~\ref{sec:examples}, we specialize this result to specific estimators, including lasso, square-root lasso, group square-root lasso, and others. In Section~\ref{sec:discussion}, we conclude with a brief discussion. The proofs are deferred to the Appendix.

The ordering of the following sections is geared towards readers that wish to dive into the technical aspects right away. For getting a first overview instead, one can have a quick glance at the model and the estimators in Displays~\eqref{eq:model} and~\eqref{eq:estimator}, respectively, and then skip directly to the examples in Section~\ref{sec:examples}.
\subsection*{Setting and Notation}\label{sec:setting}
\paragraph{Model} We consider linear regression models of the form
\begin{equation}
  \label{eq:model}
  \outcome= \design\truth +\noise
\end{equation}
with outcome $\outcome\in\rn$, design matrix $\design\in\R^{n\times p}$, regression vector $\truth\in\rp$, and noise vector $\noise\in\rn$. Our goal is prediction, that is, estimation of $\design\truth.$ We allow for general design matrices~$\design$, that is, we do not impose conditions on the correlations in~\design. Moreover, we allow for general  noise~\noise, that is, we do not restrict ourselves to specific distributions for~\noise.

\paragraph{Estimators} We are particularly interested in high-dimensional settings, where the number of parameters~$p$ rivals or even exceeds the number of samples~$n$. As needed in such settings, we assume that the regression vector~$\truth$ has some additional structure. This structure can be exploited by penalized estimators, which are the most standard methods for prediction in this context. We thus consider estimators of the form
\begin{equation*}
  \estimatort \in \argmin_{\beta\in\rp}\left\{\linkfunction\big(\,\normtwo{\outcome-\design\beta}^2\big)+\operatorname{pen}(\tuningparameter,\beta)\right\}\,,
\end{equation*}
where $\linkfunction$ is some real-valued link function, the mapping $\operatorname{pen}(\tuningparameter,\beta):\R^k\times\R^p\to\R$ accounts for the structural assumptions, and \tuningparameter\ is a vector-valued tuning parameter. More specifically, we consider assumptions on $\truth$ that are captured by (semi-)norms; the corresponding estimators then read 
\begin{equation}
  \label{eq:estimator}
  \estimatort \in \argmin_{\beta\in\rp}\big\{\linkfunction\big(\,\normtwo{\outcome-\design\beta}^2\big)+\sum_{j=1}^\numbergroups\tuningparameter_j\,\normqj{\fusedmatrix_j\beta}\big\}\,.
\end{equation}
To derive results that are specific enough to be useful in concrete examples, we impose some additional conditions in the following.

\paragraph{Link Function} The link function~$\linkfunction:\R\to[0,\infty)$ satisfies $\linkfunction(0)=0,$ is continuous and strictly increasing on $[0,\infty)$, and  is continuously differentiable on~$(0,\infty)$ with strictly positive and non-increasing derivative $\linkfunction'(x):=\frac{d}{dy}\linkfunction(y)\big\lvert_{y=x}.$ Moreover, the function $\R^n\to[0,\infty):$ $\alpha\mapsto g(\,\normtwo{\alpha}^2)$ is assumed to be strictly convex. The most important examples of link functions are $\linkfunction(x)=x$ and $\linkfunction(x)=\sqrt x.$

\paragraph{Form of the Penalties} We assume that the penalties are composite norms.  First, we assume that the tuning parameter $\tuningparameter:=(\tuningparameter_1,\dots,\tuningparameter_k)^\top$ is in $(0,\infty)^k.$  We then assume that the matrices $\fusedmatrix_1,\dots,\fusedmatrix_\numbergroups\in\R^{p\times p}$ satisfy $\bigcap_{j=1}^\numbergroups \operatorname{Ker}(\fusedmatrix_j)=\{\mathbf 0_{p\times p}\},$ where $\operatorname{Ker}$ denotes the kernel of a matrix. This assumption is mild, simply stating that the row space of all matrices $\fusedmatrix_1,\dots,\fusedmatrix_\numbergroups$ combined span~$\R^p$, that is, each parameter is covered by some penalization.  In the simplest cases, the matrices  equal the identity matrix. In general, however, these matrices allow for the incorporation of complex structural assumptions. For example, group structures can be modeled by taking the matrices $\fusedmatrix_1,\dots,\fusedmatrix_\numbergroups$ equal to (arbitrarily overlapping) block-diagonal matrices with some of the blocks in each $\fusedmatrix_j$ equal to zero. Finally, the single norms $\normqj{\cdot}$ with $q_j\geq 1$ are the regular $\ell_{q_j}$-norms on~$\R^{p}.$ Their dual norms are denoted by~$\normqjd{\cdot}$, and it holds that $\normqjd{\cdot}=\normpj{\cdot}$ for $p_j\in[1,\infty]$ such that ${1}/{p_j} + {1}/{q_j} =1.$ Since each $\normqj{\cdot}$ is a norm, and since the rows of the matrices $\fusedmatrixj$ span the entire~$\R^p$, the penalty is indeed a norm.


\paragraph{Treatment of Overlap} We introduce some further notation to make our result sharp also in cases where variables are subject to more than one penalty term, such as in the overlapping group (square-root) lasso.
For this, we first denote by $A^+$ the Moore-Penrose pseudoinverse of a matrix $A$. We then note that by the rank assumption on the matrices~$\fusedmatrix_1,\dots,\fusedmatrix_\numbergroups,$ there are projection matrices~$\projectionmatrixs\in\R^{p\times p}$ such that 
\begin{equation}
\label{eq:partition}
  \sum_{j=1}^\numbergroups\projectionmatrixj\fusedmatrixj^{+}\fusedmatrixj=\operatorname{I}_{p\times p}\,.
\end{equation}
The projection matrices enter the ``empirical process'' terms associated with the tuning parameters and  the prediction bounds. Our results hold for any $\projectionmatrixs$ that satisfy the above equality; however, appropriate choices are needed to obtain sharp bounds.  In generic examples, the choice of \projectionmatrixs\ is straightforward: if $\numbergroups=1$ (see, for example, the lasso, square-root lasso, and fused lasso) or if the row spaces of the matrices $\fusedmatrix_1,\dots,\fusedmatrix_\numbergroups$ are disjoint (see, for example, the group lasso with non-overlapping groups), one can select $\projectionmatrixs=\operatorname{I}_{p\times p}.$ More generally, if $\numbergroups>1$ and some variables are penalized twice (see, for example, the group lasso with overlapping groups), slightly more complicated choices lead to  optimal bounds.

\paragraph{Technical Assumption on the Noise Distributions} We consider general noise distributions; for example, we allow for heavy-tailed noise and for correlations within $\noise$ and between $\noise$ and $\design$. However, we exclude non-generic noise distributions for technical ease. More specifically, we assume that $\outcome\neq \bold{0}_n$ and $\min_{j\in\{1,\dots,\numbergroups\}}\norm{(\design \projectionmatrix_j\fusedmatrix_j^{+})^\top\noise}_{q_j}^*>0$ with probability one. This implies in particular that $\linkfunction'(\,\normtwo{Y-X\estimatort}^2)>0$ with probability one, see Lemma~\ref{re:nonzero} in the Appendix. To illustrate that the assumptions hold in generic cases, note that the $\norm{\cdot}_{q_j}^*$'s are norms, so that the second condition simply states that $(\design \projectionmatrix_j\fusedmatrix_j^{+})^\top\noise\neq {\bf 0}_p$, $j\in\{1,\dots,\numbergroups\},$ with probability one. As an example, consider now the non-overlapping group penalty, which corresponds to standard group lasso/square-root lasso. One can then check readily that the  assumption is equivalent to $\design^\top\noise\neq {\bf 0}_p$ holding with probability one, which is satisfied for any generic continuous distributions of $\design$ and~$\noise$. It is also straightforward to relax the condition to hold only with probability $1-\kappa,$ $\kappa\equiv \kappa(n)\to 0$ as $n\to\infty,$ to include discrete noise via standard concentration bounds; we omit the details.\label{endnotesign}

\subsection*{Relations to Existing Literature}
Statistical guarantees for high-dimensional prediction with penalized estimators  are typically formulated in terms of oracle inequalities. A variety of oracle inequalities is known, we refer to~\cite{Buhlmann13b,Buhlmann11,Hastie15} and references therein, and for some cases, also corresponding concentration inequalities are available~\cite[Theorem~1.1]{Chatterjee14}. However, most of these bounds impose severe constraints on the model, such as eigenvalue-type conditions on the design matrix~\cite{Bunea_Tsybakov_Wegkamp07,vandeGeer07} --- see~\cite{ArnakMoYo14,Hebiri13} for in-depth comparisons of \pbos\ and \sbos\ in the lasso case. In strong contrast, we are interested in oracle inequalities that do not involve additional constraints. Such results are known both for the penalized formulation of the lasso~\cite{MM11,Kolt10,RigTsy11} and for the constraint formulation~\cite{Chatterjee13}. We are interested in studying whether this type of guarantees can be established more generally in high-dimensional regression.

To achieve this generality, we introduce arguments mainly based on convexity and continuity. \label{convexityintro} Convexity has been used previously to establish $\ell_\infty$- and support recovery guarantees for the lasso~\cite{Wainwright09}, prediction bounds for the constraint version of the lasso~\cite{Chatterjee13} and for the lasso in transductive and semi-supervised learning~\cite{Bellec16b},  guarantees for low-rank matrix completion~\cite{Kolt10}, and bounds for sparse recovery via entropy penalization~\cite{Koltchinskii09c}. The arguments  in~\cite{Wainwright09} are different from ours in that they have different objectives and lead to stringent assumptions on the design. Intermediate results in~\cite[Page~6]{Chatterjee13} for the constraint lasso can be related to parts of our approach when our proof is specialized to the penalized lasso (cf. our Pages~\pageref{page:argument}-\pageref{page:last}); yet, the strategy in~\cite{Chatterjee13} based on a projection argument is different from ours, and extensions of that argument to the penalized lasso and, more generally, to our framework with multiple tuning parameters/penalty terms and different link functions do not seem straightforward. 
Finally, the convexity arguments in~\cite[Proof of Lemma~1]{Bellec16b}, \cite[Equation~(2.6)]{Kolt10}, and \cite[Inequalities~(3.1) and~(3.2)]{Koltchinskii09c} can be related to some of the techniques on Pages~\pageref{page:argument}-\pageref{page:last}.

Our continuity arguments evolve around Brouwer's fix-point theorem. As intermediate steps, we show that suitable tuning parameters exist in the first place and that the unbounded set $(0,\infty)^\numbergroups$ of tuning parameters can be replaced by a bounded set. These facts are known for the lasso and the square-root lasso, see~\cite{Belloni11} and others, but they are novel and non-trivial in the general case and might thus be of interest by themselves.

Our result specializes correctly and confirms existing expectations.
As one example, our bounds for the penalized version of the lasso match the corresponding bounds in the literature~\cite{Kolt10,RigTsy11} and relate to similar expectation-type results~\cite{MM11}.
As another example, the bounds for the (group) square-root lasso match the results of the (group) lasso, complementing previous findings that show the correspondence of the two methods in oracle inequalities under additional constraints~\cite{Belloni11,Yoyo13}.

We also mention that our bounds are near rate-optimal in the absence of further assumptions. Indeed, it has been shown that the rates for lasso prediction \cite[Proposition~4]{ArnakMoYo14} (even when the noise is Gaussian) can not be improved in general beyond $1/\sqrt n$, which corresponds to our bound up to log-factors --- see the Examples section. Under RE-type assumptions~\cite{Sara09}, one can find the rate ${s\log p}/{n}$ for the lasso prediction error, where $s$ is the number of non-zero elements in $\truth$. In the case where $s$ is small, this can be a substantial improvement over the $\sqrt{\log p/n}\,\normone{\betatrue}$-rate. Refined versions, allowing for a potentially large number of small entries in \truth, can be found in~\cite{vdGeer11}. However, RE-type assumptions are very restrictive and often seem unlikely to hold in practice. We come back to this issue in the Discussion section.

We finally relate to oracle inequalities for objectives different from prediction. Besides prediction, standard goals include variable selection and $\ell_1$-, $\ell_2$-, and $\ell_\infty$-estimation, we refer again to~\cite{Buhlmann11,Hastie15} and references therein. These objectives necessarily involve strict assumptions on the design and are thus not of major relevance here. More closely related to our work is out-of-sample prediction, which --- on a high level --- can be thought of as being somewhere between estimation and prediction. Classical results demonstrate that the  lasso can achieve out-of-sample prediction without further assumptions on the design~\cite{Greenshtein04}. Another similarity to prediction is that consistency guarantees for constraint lasso in out-of-sample prediction can be formulated in terms of $\ell_1$-balls of predictors~\cite[Theorems~1 and~3]{Greenshtein04}. (The $\ell_1$-balls in~\cite{Greenshtein04} are, however, more restrictive than the $\ell_1$-balls needed for prediction). On the other hand, the optimality of the $\ell_1$-related rates suggested in~\cite{Greenshtein04} and the applicability of the results to other estimators considered in our paper remain open questions.

\newcommand{\tuningrelax}{\textcolor{black}{\ensuremath{c}}}
\section{General Result}\label{sec:general}
We now state the general oracle inequality for the framework described on Pages~\pageref{sec:setting}--\pageref{endnotesign}. For this, we first have to discuss the existence of suitable tuning parameters. The valid tuning parameters in standard results for the lasso, for example, are of the form $\tuningrelax\,\normsup{\design^\top\noise},$ where the factor $\tuningrelax\geq 2$ depends on the specific type of oracle inequality ($\tuningrelax=2$ for the standard \pbos; $\tuningrelax>2$ for the standard sparsity prediction or estimation bounds, see~\cite{Bickel09,Chichignoud_Lederer_Wainwright14} and others). To show that there is a corresponding range for the tuning parameters in our general bounds, we derive the following result.
\begin{lemma}[Existence]\label{re:existence}
Consider fixed constants $\tuningrelax_1,\dots,\tuningrelax_\numbergroups\in(0,\infty)$. With probability one, there is a tuning parameter $\tuningparameter\equiv\tuningparameter(\tuningrelax_1,\dots,\tuningrelax_\numbergroups)\in (0,\infty)^\numbergroups$ such that
\begin{equation*}
\frac{\tuningparameter}{2g'(\,\normtwo{\outcome-\design\estimatort}^2)}=\left(\tuningrelax_1\,\norm{(\design  \projectionmatrix_1\fusedmatrix_1^{+})^\top \noise}_{q_1}^*,\dots,\tuningrelax_\numbergroups\,\norm{(\design  \projectionmatrix_\numbergroups\fusedmatrix_\numbergroups^{+})^\top \noise}_{q_\numbergroups}^*\right)^\top.
\end{equation*}
\end{lemma}
\noindent This proof of existence ensures that  suitable tuning parameters exist for arbitrary estimators of the form~\eqref{eq:estimator}. If $g:x\mapsto x,$ Lemma~\ref{re:existence} can be verified easily. In particular, the above equation simplifies to $\tuningparameter=2\tuningrelax_1\,\normsup{\design^\top\noise}$ for the lasso. In general, however, the statement is more intricate, and there might be several tuning parameters that satisfy the equality.  
Our proof is, therefore, more involved, invoking continuity arguments and Brouwer's fixed-point theorem, see Appendix.


One can also replace the implicit equalities in Lemma~\ref{re:existence} by explicit bounds on the tuning parameters. Note first that $g'(\,\normtwo{\outcome-\design\estimatort}^2)$ is not monotone in the tuning parameter in general. However, this problem can be circumvented by deriving concentration bounds. More specifically,  for any link function~\linkfunction\ and even for heavy-tailed noise distributions, one can use empirical process theory (such as~\cite{YoyoSara12}, for example) to derive concentration bounds for $|\linkfunction'\big(\,\normtwo{\outcome-\design\estimatort}^2\big)- \linkfunction'\big(\,\normtwo{\noise}^2\big)|$ (very rough bounds are sufficient). This then implies bounds for the left-hand side in Lemma~\ref{re:existence}, and consequently, allows one to replace the implicit inequalities by explicit lower bounds on the tuning parameters.

\pagebreak

We are now ready to state the main result.
\begin{theorem}[\Pbo]\label{r:finalbound}
For any choice of $\tuningrelax_1,\dots,\tuningrelax_\numbergroups\in(0,\infty)$, with probability one, the estimator~$\estimatort$ defined by~\eqref{eq:estimator} with tuning parameter $\tuningparameter$ as in Lemma~\ref{re:existence} above satisfies the prediction bound
 \begin{multline*}
\frac{1}{n}\,\normtwo{\design(\betatrue-\estimatort)}^2\le\inf_{\substack{\genconstant\in(0,1)\\\genbeta\in\R^p}} \Big\{\frac{1}{4\genconstant(1-\genconstant)n}\,\normtwo{\design(\betatrue-\genbeta)}^2\\+\frac{1}{n}\sum_{j=1}^{\numbergroups}\frac{1+\tuningrelax_j}{1-\genconstant}\,  \normqjd{(\design  \projectionmatrix_j\fusedmatrix_j^{+})^\top\noise}  \, \normqj{\fusedmatrix_j \genbeta}- \frac{1}{n}\sum_{j=1}^{\numbergroups}\frac{\tuningrelax_j-1}{1-\genconstant}\, \normqjd{(\design  \projectionmatrix_j\fusedmatrix_j^{+})^\top\noise} \,  \normqj{\fusedmatrix_j \estimatort}\Big\}\,.
\end{multline*}
\end{theorem}
\noindent  This oracle inequality provides bounds for the prediction errors of the estimators~\eqref{eq:estimator}. Our proofs are based only on convexity and continuity arguments, making the result very general and sharp in its constants.

Let us provide some interpretations of the result. 
Note first that the bounds apply to any positive $\tuningrelax_1,\dots,\tuningrelax_\numbergroups,$ but the most interesting case is $\tuningrelax_1,\dots,\tuningrelax_\numbergroups\geq 1$, since the terms on the right-hand side that depend on the estimator itself can then be dropped. 
We first consider a tuning parameter $\lo\in (0,\infty)^\numbergroups$  that satisfies the equality in Lemma~\ref{re:existence} with $\tuningrelax_1,\dots,\tuningrelax_\numbergroups=1$.
We set $\estimatororacle:=\est^{\lo}$ for ease of notation. Now, choosing $\genconstant=0.5$ in the above functional implies
\begin{equation}\label{special1}
  \frac{1}{n}\,\normtwo{\design(\betatrue-\estimatororacle)}^2\le\min_{\substack{\genbeta\in\R^p}} \Big\{\frac{1}{n}\,\normtwo{\design(\betatrue-\genbeta)}^2+\frac{4}{n}\sum_{j=1}^{\numbergroups}\,  \normqjd{(\design  \projectionmatrix_j\fusedmatrix_j^{+})^\top\noise}  \, \normqj{\fusedmatrix_j \genbeta}\Big\}\,.
\end{equation}
To bring this on a more abstract level, we denote the prediction loss by $\Loss(\genbeta):=\,\normtwo{\design(\betatrue-\genbeta)}^2/n$ and introduce model classes $\classv:=\{\genbeta\in\R^p:4\sum_{j=1}^{\numbergroups}\,  \normqjd{(\design  \projectionmatrix_j\fusedmatrix_j^{+})^\top\noise}  \, \normqj{\fusedmatrix_j \genbeta}/n=\constv\}$ indexed by~$\constv\in[0,\infty).$ Inequality~\eqref{special1} then reads
\begin{equation*}
  \Loss(\estimatororacle)\leq\min_{\constv\in[0,\infty)}\min_{\genbeta\in\classv}\{\Loss(\genbeta)+\constv\}\,.
\end{equation*}
Thus, the estimator $\estimatororacle$ performs as well as the minimizer of the loss over the class~$\classv$ --- up to a complexity penalty of~\classv. For another abstract view on the theorem, we define a loss for any $\Riskconst\in(0,\infty)^\numbergroups$ by
\begin{equation*}
  \Lossa(\beta):=\frac{1}{n}\,\normtwo{\design(\betatrue-\genbeta)}^2 +\frac{1}{n}\sum_{j=1}^{\numbergroups}\riskconst_j \,  \normqj{\fusedmatrix_j \genbeta}\,.
\end{equation*}
This loss balances prediction accuracy against model complexity.
Now, we observe that if $2(\tuningrelax_j-1)\,\normqjd{(\design\projectionmatrix_j\fusedmatrix_j^{+})^\top\noise}=\riskconst_j,$ choosing again $\genconstant=0.5$ in the initial functional yields
\begin{equation*}
  \Lossa(\estimatort)\leq\Big(1+\max_{j\in\{1,\dots,\numbergroups\}}\frac{4\,\normqjd{(\design\projectionmatrix_j\fusedmatrix_j^{+})^\top\noise}}{\riskconst_j}\Big) \min_{\genbeta\in\R^p}\Lossa(\genbeta)\,.
\end{equation*}
This means that the estimator $\estimatort$ performs as well --- again up to constants --- as the minimizer of the loss $\Lossa.$ 
These forms of our bound fit the classical notions of oracle inequalities (with sharp leading constant) in empirical risk minimization~\cite[Chapter~1.1]{koltch11}  and non-parametric estimation~\cite[Chapter~1.8]{Tsybakov09}.
Finally, we consider again~$\lo,$ then set $\genbeta=\betatrue$ and take the limit $\genconstant\to0$ in Theorem~\ref{r:finalbound}. We find
\begin{equation}\label{special2}
  \frac{1}{n}\,\normtwo{\design(\betatrue-\estimatororacle)}^2\le\frac{2}{n}\sum_{j=1}^{\numbergroups}\,  \normqjd{(\design  \projectionmatrix_j\fusedmatrix_j^{+})^\top\noise}  \, \normqj{\fusedmatrix_j \betatrue}\,.
\end{equation}
This form of our bound commensurates with typical formulations of oracle inequalities in high dimensions~\cite[Chapters~2.4.2 and 6.2]{Buhlmann11}; in particular, the bound implies known prediction bounds with correct constants --- see the following section.

The tuning parameter $\lo$ minimizes  $\sum_{j=1}^{\numbergroups}(1+\tuningrelax_j)\,  \normqjd{(\design  \projectionmatrix_j\fusedmatrix_j^{+})^\top\noise}  \, \normqj{\fusedmatrix_j \genbeta}/(n(1-\genconstant)),$ the first term on the right-hand side of the bound, as a function of $\tuningrelax_1,\dots,\tuningrelax_\numbergroups$ under the constraint that $\tuningrelax_1,\dots,\tuningrelax_\numbergroups\geq 1$ (which implies that the  terms on the right-hand side that depend on the estimator can be dropped).  
This choice of the tuning parameter also leads to  rates that have been shown to be near-optimal in certain cases \cite[Proposition~4]{ArnakMoYo14}. 
However, this does not necessarily mean that $\lo$ minimizes the prediction loss $\normtwo{\design(\betatrue-\estimatort)}^2.$ 
Explicit formulations of tuning parameters that minimize the prediction loss are to date unknown. 
Nevertheless, some insights have been developed: for example, \cite{ArnakMoYo14,Hebiri13} discuss  lasso tuning parameters as a function of the correlations in $\design,$ and \cite{Bellec16} and \cite[Section~4]{Sun13} discuss lasso/scaled lasso tuning parameters that can lead to minimax rates in the case of sparsity and small correlations. 
Furthermore, our essay does not provide any guidance on how to select tuning parameters in practice; 
indeed, the tuning parameters in Lemma~\ref{re:existence}  depend on the noise~$\noise$, which is unknown in practice. 
For ideas on the practical selection of the lasso tuning parameter with finite sample guarantees, we refer to~\cite{Chetelat_Lederer_Salmon14,Chichignoud_Lederer_Wainwright14}.
For ideas on how to make the selection of tuning parameters independent of unknown model aspects, we refer to \cite{Giraud12,Lederer14} and the square-root/scaled lasso example in the following section.

We conclude this section highlighting five other properties of Theorem~\ref{r:finalbound} (much of this becomes more lucid in the context of the specific examples discussed in the next section): First, the bound involves the values of the tuning parameters to the power one and holds for any design matrix~\design. The bound contains the penalty values of the regression vectors. Hence, the bounds are \pbos. Second, the oracle inequality holds for any distribution of the noise~\noise. Third, the bounds hold for any sample size $n$; in particular, the bounds are non-asymptotic. Fourth, the bounds become smaller if the correlations in~$\design$ become larger, cf.~\cite{Hebiri13}. Fifth, the link function~$\linkfunction$ appears in the existence result on tuning parameters but not in the prediction bound. This last, interesting point clarifies the role of the link function: its essential purpose is to ``reshuffle'' the tuning parameter path. One can relate this observation to the discussion of the lasso/square-root lasso below.

\section{Examples}\label{sec:examples}
\newcommand\iid{\stackrel{\mathclap{\normalfont\mbox{\scriptsize iid}}}{\sim}}
We now apply our general results to some specific estimators of the form~\eqref{eq:estimator}.

\paragraph{Lasso} The lasso~\cite{Tibshirani96} is defined as
\begin{equation*}
  \estimatort \in \argmin_{\beta\in\rp}\{\,\normtwo{\outcome-\design\beta}^2+\tuningparameter\,\normone{\beta}\}\,.
\end{equation*}
We first show that we can recover the standard penalty bounds that have been derived in the lasso literature, see, for example,~\cite[Eq. (2.3) in Theorem 1]{Kolt10} and~\cite[Equation~(3)]{Hebiri13}. The proofs use that lasso objective function is minimal at~\estimatort, that is, for any $\genbeta\in\R^p,$ it holds that
\begin{equation*}
  \normtwo{\outcome-\design\estimatort}^2+\tuningparameter\,\normone{\estimatort}\leq \normtwo{\outcome-\design\genbeta}^2+\tuningparameter\,\normone{\genbeta}\,.
\end{equation*}
This is equivalent to 
\begin{equation*}
  \normtwo{\outcome-\design\betatrue+\design\betatrue-\design\estimatort}^2+\tuningparameter\,\normone{\estimatort}\leq \normtwo{\outcome-\design\betatrue+\design\betatrue-\design\genbeta}^2+\tuningparameter\,\normone{\genbeta}
\end{equation*}
and
\begin{multline*}
  \normtwo{\outcome-\design\betatrue}^2+2\langle\outcome-\design\betatrue,\design\betatrue-\design\estimatort\rangle+\normtwo{\design\betatrue-\design\estimatort}^2+\tuningparameter\,\normone{\estimatort}\\
\leq \normtwo{\outcome-\design\betatrue}^2+2\langle\outcome-\design\betatrue,\design\betatrue-\design\genbeta\rangle+\normtwo{\design\betatrue-\design\genbeta}^2+\tuningparameter\,\normone{\genbeta}\,.
\end{multline*}
Invoking the model~\eqref{eq:model} and consolidating, this yields
\begin{equation*}
    \normtwo{\design\truth-\design\estimatort}^2\leq \normtwo{\design\betatrue-\design\genbeta}^2+2\inprod{\noise}{\design\estimatort-\design\genbeta}+\tuningparameter\,\normone{\genbeta}-\tuningparameter\,\normone{\estimatort}\,.
\end{equation*}
H\"older's inequality and the triangle inequality then lead to
\begin{equation*}
    \normtwo{\design\truth-\design\estimatort}^2\leq \normtwo{\design\betatrue-\design\genbeta}^2+2\,\normsup{\design^\top\noise}(\,\normone{\estimatort}+\normone{\genbeta})+\tuningparameter\,\normone{\genbeta}-\tuningparameter\,\normone{\estimatort}\,.
\end{equation*}
Hence, for $\lambda=\lo=2\,\normsup{\design^\top\noise},$ we eventually find
\begin{equation*}
   \frac{1}{n}\, \normtwo{\design(\truth-\estimatororacle)}^2\leq\min_{\genbeta\in\R^p}\Big\{\, \frac{1}{n}\,\normtwo{\design(\betatrue-\genbeta)}^2+\frac{4}{n}\,\normsup{\design^\top\noise}\,\normone{\genbeta}\Big\}\,.
\end{equation*}
Observing that $\numbergroups=1$ and $\fusedmatrix_1=\projectionmatrix_1=\operatorname{I}_{p\times p}$ in this example, one can check that \eqref{special1} recovers this bound.

In the case $\genbeta=\betatrue,$ Inequality~\eqref{special2} does slightly better. Indeed, our results imply
\begin{equation*}
\frac{1}{n}\,\normtwo{\design(\betatrue-\estimatororacle)}^2 \le \frac 2 n\,\normsup{\design  ^\top\noise}\,   \normone{\betatrue}\,,
\end{equation*} 
which is smaller by a factor $2$ than the above right-hand side at $\genbeta=\betatrue.$ 
The same bound also follows from~\cite[Inequality~(23)]{ScaledLasso11}.
The key property that allows one to derive bounds with the improved factor is convexity; therefore, in the lasso case, one can also find the above bound with the techniques in the papers mentioned in the corresponding discussion on Page~\pageref{convexityintro}, such as~\cite{Bellec16b,Kolt10,Koltchinskii09c}.

To provide a sense for the rates, we mention that if $\noise_1,\dots,\noise_n\iid\mathcal N(0,\sigma^2) $ and $(\design^\top\design)_{jj}=n$ for all~$j\in\{1,\dots,p\}$, it holds that $\lo\approx \sigma\sqrt{n\log(p)}$ and $\normtwo{\design(\betatrue-\estimatororacle)}^2/n \lesssim \sigma\sqrt{\log(p)/n}\ \normone{\betatrue}$.\footnote{The wiggles indicate that we are interested only in the rough shapes and neglect constants, for example.} Importantly, the rate is lower bounded by $1/\sqrt n$ --- unless further assumptions on the design matrix~\design\ are imposed~\cite[Proposition~4]{ArnakMoYo14}.

We finally note that one can also include ``tailored'' tuning parameters, which corresponds to considering the lasso as a specification of our general framework with $\numbergroups=p$ and not necessarily equal tuning parameters $\tuningparameter_1,\dots,\tuningparameter_p.$ One can check easily that in particular if $(\design^\top\noise)_1,\dots,(\design^\top\noise)_p$ i.i.d, one obtains the same bounds as above.

\paragraph{Square-root/scaled lasso} The square-root lasso~\cite{Belloni11}  reads
\begin{equation*}
  \estimatort \in \argmin_{\beta\in\rp}\left\{\,\normtwo{\outcome-\design\beta}+\tuningparameter\,\normone{\beta}\right\}\,.
\end{equation*}
In our framework, $\numbergroups=1$ and $\fusedmatrix_1=\projectionmatrix_1=\operatorname{I}_{p\times p},$ so that $\lo=\,\normsup{\design^\top\noise}/\,\normtwo{\outcome-\design\estimator^{\lo}}$ and the prediction bound~\eqref{special2} reads
\begin{equation*}
\frac{1}{n}\,\normtwo{\design(\betatrue-\estimatororacle)}^2 \le \frac 2 n\,\normsup{\design  ^\top\noise}\,   \normone{\betatrue}\,.
\end{equation*}
A similar bound is implied by~\cite[Theorem~1]{ScaledLasso11}, and Inequality~\eqref{special1} and the general bound in our main theorem translate accordingly (for ease of comparison, we focus on \eqref{special2} in the following). The bounds match the corresponding ones for the lasso, but the tuning parameters differ. The crux of the square-root lasso, and similarly, the scaled lasso~\cite{ScaledLasso11}, is that their tuning parameters can be essentially independent of the noise variance. For example,  if $\noise_1,\dots,\noise_n\iid\mathcal N(0,\sigma^2)$ and $(\design^\top\design)_{jj}=n$ for all~$j\in\{1,\dots,p\}$, it holds\footnote{see Page~\pageref{sqrtexproof} for some hints} that $\lo\approx \normsup{\design^\top\noise}/\,\normtwo{\noise}\approx\sqrt{\log(p)}$, \label{sqrtex} which is independent of $\sigma.$ Since $\sigma$ is typically unknown in practice, the square-root/scaled lasso can thus facilitate the tuning of $\tuningparameter$.

\paragraph{Slope estimator} The slope estimator~\cite{Bogdan15} can be written as
\begin{equation*}
  \estimatort \in \argmin_{\beta\in\rp}\Big\{\,\normtwo{\outcome-\design\beta}^2+\tuningparameter\sum_{j=1}^p\omega_j|\beta|_{(j)}\Big\}\,,
\end{equation*}
where $|\beta|_{(j)}$ denotes the $j$th largest entry of $\beta$ in absolute value, $\omega_1\geq \dots \geq \omega_p>0$ is a non-increasing sequence of weights, and $\tuningparameter>0$ is a tuning parameter.  A promising case for the slope estimator is the one where $\noise_1,\dots,\noise_n\iid\mathcal N(0,\sigma^2) $ and $(\design^\top\design)_{jj}=n$ for all~$j\in\{1,\dots,p\}.$ The weights can then be chosen as $\omega_j:=2\sigma\sqrt{n\log(2p/j)}$ in the spirit of the Benjamini-Hochberg procedure~\cite{Bogdan15,Su16}, and  a theoretically justified choice of the tuning parameter is $\lambda>4+\sqrt 2$~\cite[Equation~(2.5)]{Bellec16}. In particular, this choice works in the sense of the first part of  Theorem~\ref{r:finalbound}, and one finds the bound  $\,\normtwo{\design(\betatrue-\estimatororacle)}^2/n \lesssim \tuningparameter \sigma\sqrt{\log(p)/n}\ \normone{\betatrue}$, for example, which coincides with the bounds above. Similar considerations apply to the oscar penalty~\cite{Bondell08}.


\paragraph{Elastic net} The elastic net~\cite{Zou05} reads
\begin{equation*}
  \estimatort \in \argmin_{\beta\in\rp}\left\{\,\normtwo{\outcome-\design\beta}^2+\tuningparameter_1\,\normone{\beta}+\tuningparameter_2\,\normtwo{\beta}^2\right\}\,.
\end{equation*}
This is not directly in the form~\eqref{eq:estimator}. However, one can use the usual trick writing the estimator as a lasso with augmented data, cf.~\cite[Lemma~1]{Hebiri11}. Using $\fusedmatrix_1=\projectionmatrix_1=\operatorname{I}_{p\times p}$, our results then hold for any  tuning parameters that satisfy $\tuningparameter_1=2\,\normsup{\design^\top\noise-\tuningparameter_2\betatrue}$. For example, we can set $\lo_2=\argmin_{\tuningparameter_2}\normsup{\design^\top\noise-\tuningparameter_2\betatrue}$ and $\lo_1=2\,\normsup{\design^\top\noise-\lo_2\betatrue}$. The bound in~\eqref{special2}, for example, then reads
\begin{equation*}
\frac{1}{n}\,\normtwo{\design(\betatrue-\estimatororacle)}^2\le \frac{2}{n}\, \normsup{\design^\top\noise-\lo_2\betatrue}\,\normone{\betatrue} \le \frac{2}{n}\,\normsup{\design  ^\top\noise}\,   \normone{\betatrue}\,.
\end{equation*} 
Similar results hold if, for example, the settings with the normal noise vectors described in the two examples above apply and $\lo_2=\mathcal O(\sqrt n).$ The main intent of the elastic net is to improve variable selection. However, our results show that the elastic net with well-chosen tuning parameters also has similar penalty guarantees for prediction as the lasso.

\paragraph{Lasso and square-root lasso with group structures} The estimators considered so far are based on a simple notion of sparsity. In practice, however, it can be reasonable to assume more complex sparsity structures in the regression vector~$\betatrue.$ Estimators that take such structures into account include the group lasso~\cite{YuanLin06}, group square-root lasso~\cite{Yoyo13}, hierarchical group lasso~\cite{Bien13}, lasso with overlapping groups~\cite{Jacob09}, and sparse group lasso~\cite{Simon13}. They all fit our framework. 

In contrast to the examples above, the matrices $\fusedmatrix_{1},\dots,\fusedmatrix_{\numbergroups},\projectionmatrix_1,\dots,\projectionmatrix_\numbergroups$ play a nontrivial role in these examples. On a high level, the matrix~$\fusedmatrix_{j}$ specifies which variables are incorporated in the $j$th group. If groups overlap, the matrix~$\projectionmatrix_j$ specifies which parts of $\design^\top \noise$ are attributed to the  tuning parameter~$\lo_j.$ For example, if the $m$th variable is in the $j$th and $l$th group, the matrices~$\projectionmatrix_j$ and~$\projectionmatrix_l$ can be chosen such that the corresponding element $(\design^\top\noise)_m$ is part of either $\lo_j$ or $\lo_l$ and not in both of them.

As an illustration, let us consider the group lasso with non-overlapping groups:
\begin{equation*}
  \estimatort \in \argmin_{\beta\in\rp}\Big\{\,\normtwo{\outcome-\design\beta}^2+\tuningparameter\sum_{j=1}^\numbergroups\, \normtwo{\beta_{G_j}}\Big\}\,.
\end{equation*}
Here, $G_1,\dots,G_\numbergroups$ is a partition of $\{1,\dots,p\}$ and $(\beta_{G_j})_i:=\beta_i\1\{i\in G_j\}.$ To put the estimator in  framework~\eqref{eq:estimator}, one can either generalize the estimator to incorporate possibly different tuning parameters for each group, set $\fusedmatrix_j$ to~$(\fusedmatrix_j)_{st}=\1\{s=t,s\in G_j\}$, and then choose a dominating tuning parameter, or one can directly extend our results (as mentioned earlier) to arbitrary norm penalties. In any case, the projection matrices/matrix equal the identity matrix, our standard tuning parameter  is $\lo=2\max_{l\in\{1,\dots,\numbergroups\}}\,\normtwo{(\design^\top\noise)_{G_l}}$, and the bound~\eqref{special2} gives  
\begin{equation*}
\frac{1}{n}\,\normtwo{\design(\betatrue-\estimatororacle)}^2 \le \frac{2}{n}\max_{l\in\{1,\dots,\numbergroups\}}\,\normtwo{(\design^\top\noise)_{G_l}} \sum_{j=1}^\numbergroups\,\normtwo{\betatrue_{G_j}} \,.
\end{equation*} 

\paragraph{Trend filtering and total variation/fused penalty} Trend filtering~\cite{Kim09,Tibshirani14} reads
\begin{equation*}
  \estimatort \in \argmin_{\beta\in\rp}\big\{\,\normtwo{\outcome-\beta}^2+\tuningparameter\, \normone{\fusedmatrix\beta}\big\}\,,
\end{equation*}
where  for given $l\in\{1,2,\dots\}$, the matrix~$\fusedmatrix\in\R^{p\times p}$ is defined as $\fusedmatrix:=\underbrace{D\times \cdots \times D}_{l\text{~times}}$ with $D\in\R^{p\times p}$ given by
\begin{align*}
 D_{ij}:=\begin{cases}
\textcolor{black}{-}1&\text{~~~if~}i<p\text{~and~}i=j\\
\textcolor{white}{-}1&\text{~~~if~}i<p\text{~and~}i=j-1\\
\textcolor{white}{-}0&\text{~~~otherwise}
\end{cases}\,.
\end{align*}
We find $\lo=2\,\normsup{\fusedmatrix^+{}^\top\noise}$, and~\eqref{special2} implies\footnote{Unlike assumed earlier, $\operatorname{Ker}(\fusedmatrix)\neq \{\mathbf 0_{p\times p}\}$ in this example, but one can replace matrix $\fusedmatrix$ by the invertible matrix $\fusedmatrix+\epsilon\operatorname{I_{p\times p}}$ and then take the limit $\epsilon\to 0.$}
\begin{equation*}
\frac{1}{n}\,\normtwo{\betatrue-\estimatororacle}^2 \le \frac 2 n\,\normsup{\fusedmatrix^+{}^\top\noise}\,   \normone{\fusedmatrix\betatrue}\,.
\end{equation*}
In the case $l=1,$ the estimator becomes 
\begin{equation*}
  \estimatort \in \argmin_{\beta\in\rp}\big\{\,\normtwo{\outcome-\beta}^2+\tuningparameter\sum_{j=2}^p|\beta_j-\beta_{j-1}|\big\}\,,
\end{equation*}
which corresponds to the total variation~\cite{Rudin92} and fused lasso penalizations~\cite{Tibshirani05}. Moreover, one can check that the Moore-Penrose inverse of $\fusedmatrix=D$ is then given by $D^+$ with entries
\begin{align*}
 D_{ij}^+:=\begin{cases}
~(j-p)/p&\text{~~~if~}i\leq j<p\\
~j/p&\text{~~~if~}i> j,~j<p\\
~0&\text{~~~if~}j=p
\end{cases}\,.
\end{align*}
We find  $\lo=2\,\normsup{D^+{}^\top\noise}$ and the corresponding bound
\begin{equation*}
\frac{1}{n}\,\normtwo{\betatrue-\estimatororacle}^2 \le \frac 2 n\,\normsup{D^+{}^\top\noise}\,   \normone{D\betatrue}\,.
\end{equation*}

\section{Discussion}\label{sec:discussion}
\Sbos\ have been derived for many high-dimensional estimators. In this paper, we complement these bounds with corresponding \pbos. Which type of bound is sharper depends on the underlying model. As a general rule, \pbos\ improve with increasing correlations in the design matrix, while \sbos\ deteriorate with increasing correlations and are eventually infinite once the design matrix is too far from an orthogonal matrix~\cite{ArnakMoYo14}.

Without making assumptions on the design, and for a wide range of penalized estimators, our results imply non-trivial rates of convergence for prediction.  This is of direct practical relevance, since  the assumptions inflicted by \sbos\ are often unrealistic in applications and, in any case, depend on inaccessible model parameters and thus cannot be verified in practice. For example, \sbos\ for the lasso have been derived under a variety of assumptions  on $\design$, including RIP, restricted eigenvalue condition, and compatibility condition, see~\cite{Sara09} for an overview of these concepts. Results from random matrix theory show that these assumptions are fulfilled with high probability if the data generating process is ``nice'' (sub-Gaussian, isotropic, ...) and the sample size~$n$ is large enough, see~\cite{Sara14} for a recent result. 
Unfortunately, in practice, the data generating processes are not necessarily nice, and sample sizes can be small --- not only in comparison with the number of parameters, but also in absolute terms. 
Moreover, even if the conditions of fast rates bounds are satisfied, these bounds can contain very large factors and are then only interesting from an asymptotic point of view.

\section*{Acknowledgements} We thank Jacob Bien and Mohamed Hebiri for their insightful comments and the Editors and Reviewers for their helpful suggestions.

\bibliography{References}
\pagebreak
\appendix
\section{Proofs}\label{sec:proofs}

We start with four auxiliary results, Lemmas~\ref{re:convexity}-\ref{re:nonzero}. 
We then prove Lemma~\ref{re:existence} and Theorem~\ref{r:finalbound}. Figure~\ref{fig:dependence} depicts the dependence structure of the results.\\

\tikzstyle{block} = [rectangle, draw=black, fill=white, node distance=1.1cm, minimum height=2em, font=\color{black}\sffamily, rounded corners=5pt, text width=5em, text badly centered]
\tikzstyle{cloud} = [ellipse, draw=black, fill=white, node distance=1.1cm, minimum height=2em, font=\color{black}\sffamily,text width=4em, text badly centered]
\tikzstyle{line} = [draw, -latex']

\begin{figure}[h]    
\begin{tikzpicture}[auto, node distance=1.5cm]
    \node [cloud] (A1) {Lemma~\ref{re:convexity}};
    \node [cloud, right = of A1] (A2) {Lemma~\ref{re:functions}};
    \node [cloud, right = of A2] (A3) {Lemma~\ref{re:nonzero}};
    \coordinate (Middle) at ($(A1)!0.5!(A2)$);
     \coordinate (Middlep) at ($(A2)!0.5!(A3)$);
     \node [cloud, below = of Middle] (Lemma) {Lemma~\ref{re:existence}};
    \node [block, below = of Middlep] (Thm) {Theorem~\ref{r:finalbound}};
    \path [line] (A1) -- (A2);
    \path [line] (A3) -- (Thm);
    \path [line] (A2)  -- (Thm);
    \path [line] (A1) -- (Lemma);
     \path [line] (A3) -- (Lemma);
    \path [line] (A2) -- (Lemma);
\end{tikzpicture}
\caption{Dependencies between the results. For example, the arrow between Lemma~\ref{re:convexity} and Lemma~\ref{re:existence} depicts that the proof of Lemma~\ref{re:existence} makes use of Lemma~\ref{re:convexity}.}
\label{fig:dependence}
\end{figure}
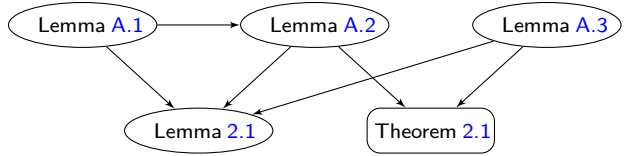
{\bf Hints for Page~\pageref{sqrtex}:}\label{sqrtexproof}
Note that (see \cite[Section~2.2]{vanderVaart96} for maximal inequalities that can be used for the last line)
\begin{align*}
\normtwo{\outcome-\design\estimatororacle}
= &\ \normtwo{\outcome-\design\truth+\design\truth-\design\estimatororacle}\\
\le&\ \normtwo{\outcome-\design\truth}+\normtwo{\design\truth-\design\estimatororacle}\\
\le& \ \normtwo{\noise}+\sqrt{\frac{2\,\normsup{\design^\top\noise}\,\normone{\truth}}{n}}\sqrt{n}\\
\lesssim& \ \sqrt{n}+\sqrt{\frac{\log(p)\ \normone{\truth}}{\sqrt n}}\sqrt{n}
\end{align*}
and similarly 
\begin{align*}
\normtwo{\outcome-\design\estimatororacle}
\gtrsim& \ \sqrt{n}-\sqrt{\frac{\log(p)\ \normone{\truth}}{\sqrt n}}\sqrt{n}\,.
\end{align*}
Thus, as long as $\normone{\truth}=o(\sqrt n/\log (p)),$ it holds that $\lo\approx \normsup{\design^\top\noise}/\,\normtwo{\noise}\approx\sqrt{\log(p)}.$ 

\subsection{Auxiliary Lemmas}

\begin{lemma}\label{re:convexity} 
For any
$\estimatort, \tildeestimator  \in \argmin_{\beta\in\R^p} \big\{g(\,\normtwo{\outcome-\design\beta}^{2})+\sum_{j=1}^{\numbergroups}\lambda_j\,\normqj{\fusedmatrix_j\beta}\big\}$ and $\alpha \in [0,1],$  it holds that $\design\estimatort=\design \tildeestimator$ and
\begin{equation*}
\alpha\estimatort+(1-\alpha)\tildeestimator \in \argmin_{\beta\in\R^p} \big\{g(\,\normtwo{\outcome-\design\beta}^{2})+\sum_{j=1}^{\numbergroups}\lambda_j\,\normqj{\fusedmatrix_j\beta}\big\}\,.
\end{equation*}
\end{lemma}

\begin{lemma}\label{re:functions}
Let $\R^\numbergroups$ be  equipped with the Euclidean norm, and $\R$ be equipped with the  absolute value norm. Then, the function
\begin{align*}
(0,\infty)^\numbergroups &\to \R\\
  \lambda&\mapsto g(\,\normtwo{\outcome-\design\estimatort}^{2})
\end{align*}
and the function
\begin{align*}
(0,\infty)^\numbergroups &\to \R\\
  \lambda&\mapsto \normtwo{\outcome-\design\estimatort}^{2}
\end{align*}
are both continuous.
\end{lemma}

\begin{lemma}\label{re:nonzero}
 With probability one, it holds that $\,\normtwo{\outcome-\design\estimatort}^2>0$ and $g'(\,\normtwo{\outcome-\design\estimatort}^2)> 0$ for any tuning parameter $\lambda\in(0,\infty)^\numbergroups$.
 \end{lemma}

\subsection{Proofs of the Auxilliary Lemmas}

\begin{proof}[Proof of Lemma~\ref{re:convexity}] 
The case $\alpha\in\{0,1\}$ is straightforward, so that we consider a given $\alpha\in(0,1).$ We first show that $\design\estimatort=\design \tildeestimator$. 
Since the function $\alpha\mapsto g(\,\normtwo{\alpha}^2)$ is strictly convex by assumption, it follows for any vectors $a,\tilde a\in\R^n$ that
\begin{align*}
g(\,\normtwo{\alpha a+(1-\alpha)\tilde a}^{2})\le \alpha g(\,\normtwo{a}^{2})+(1-\alpha)g(\,\normtwo{\tilde a}^{2})
\end{align*}
with strict inequality if $a\neq \tilde a$.  Using this with $a=\outcome-\design \estimatort$ and $\tilde a= \outcome-\design\tildeestimator,$ and invoking the convexity of the norms~$\normqj{\cdot},$ we find
\begin{align*}
&g(\,\normtwo{\outcome-\design\big(\alpha \estimatort+(1-\alpha)\tildeestimator\big)}^{2})+\sum_{j=1}^{\numbergroups}\lambda_j\,\normqj{\fusedmatrix_j\big(\alpha \estimatort+(1-\alpha)\tildeestimator\big)}\\
=&g(\,\normtwo{\alpha(\outcome-\design\estimatort)+(1-\alpha)(\outcome-\design\tildeestimator)}^{2})+\sum_{j=1}^{\numbergroups}\lambda_j\,\normqj{\alpha \fusedmatrix_j\estimatort+(1-\alpha)\fusedmatrix_j\tildeestimator}\\
\leq&\alpha g(\,\normtwo{\outcome-\design\estimatort}^{2})+(1-\alpha)g(\,\normtwo{\outcome-\design\tildeestimator}^{2})+\alpha\sum_{j=1}^{\numbergroups} \lambda_j \,\normqj{\fusedmatrix_j\estimatort}+(1-\alpha)\sum_{j=1}^{\numbergroups}\lambda_j \, \normqj{\fusedmatrix_j\tildeestimator}\\
=&\alpha \Big(g(\,\normtwo{\outcome-\design\estimatort}^{2})+\sum_{j=1}^{\numbergroups} \lambda_j \,\normqj{\fusedmatrix_j\estimatort}\Big)+(1-\alpha)\Big(g(\,\normtwo{\outcome-\design\tildeestimator}^{2})+\sum_{j=1}^{\numbergroups}\lambda_j\,\normqj{\fusedmatrix_j\tildeestimator}\Big)
\end{align*}
with strict inequality if $\design\estimatort\neq \design \tildeestimator.$ 
Moreover, we note that 
\begin{equation*}
\estimatort, \tildeestimator  \in \argmin_{\beta\in\R^p} \big\{g(\,\normtwo{\outcome-\design\beta}^{2})+\sum_{j=1}^{\numbergroups}\lambda_j\,\normqj{\fusedmatrix_j\beta}\big\}  
\end{equation*}
implies 
\begin{align*}
g(\,\normtwo{\outcome-\design\tildeestimator}^{2})+\sum_{j=1}^{\numbergroups}\lambda_j\,\normqj{\fusedmatrix_j\tildeestimator}=g(\,\normtwo{\outcome-\design\estimatort}^{2})+\sum_{j=1}^{\numbergroups}\lambda_j\,\normqj{\fusedmatrix_j\estimatort}\,.
\end{align*}
Combining the results yields
\begin{align*}
&g(\,\normtwo{\outcome-\design\big(\alpha \estimatort+(1-\alpha)\tildeestimator\big)}^{2})+\sum_{j=1}^{\numbergroups}\lambda_j\,\normqj{\fusedmatrix_j\big(\alpha \estimatort+(1-\alpha)\tildeestimator\big)}\\
\leq&g(\,\normtwo{\outcome-\design\estimatort}^{2})+\sum_{j=1}^{\numbergroups} \lambda_j \,\normqj{\fusedmatrix_j\estimatort}
\end{align*}
with strict inequality if $\design\estimatort\neq \design \tildeestimator.$ Using again that 
\begin{equation*}
  \estimatort \in \argmin_{\beta\in\R^p} \left\{g(\,\normtwo{\outcome-\design\beta}^{2})+\sum_{j=1}^{\numbergroups}\lambda_j\,\normqj{\fusedmatrix_j\beta}\right\}\,,
\end{equation*}
we find that the above inequality is actually an equality, so that 
\begin{equation*}
\design\estimatort=\design\tildeestimator\,.
\end{equation*}
and
\begin{equation*}
\alpha\estimatort+(1-\alpha)\tildeestimator\in \argmin_{\beta\in\R^p} \bigg\{g(\,\normtwo{\outcome-\design\beta}^{2})+\sum_{j=1}^{\numbergroups}\lambda_j\,\normqj{\fusedmatrix_j\beta}\bigg\}\,
\end{equation*}
as desired.

\end{proof}

\begin{proof}[Proof of Lemma~\ref{re:functions}]                  

To show that the function $\lambda\mapsto g(\,\normtwo{\outcome-\design\estimatort}^{2})$ is continuous, we first show that the function 
\begin{align*}
\R^\numbergroups &\to \R \\
\lambda &\mapsto g(\,\normtwo{\outcome-\design\estimatort}^{2})+\sum_{j=1}^\numbergroups\lambda_j\,\normqj{\fusedmatrix_j\estimatort}
\end{align*} 
is continuous.\\
For this, we show the continuity at any  fixed  $\lambda\in(0,\infty)^\numbergroups$. Given an $\epsilon>0,$ define  $m:=2\max_{1\le j\le \numbergroups}\left\{\frac{g(\,\normtwo{\outcome}^2)}{\lambda_j}\right\}$ and $\delta:= \min_{1\le j\le \numbergroups}\frac{\lambda_j}{2}\wedge \frac{\epsilon}{\sqrt{k}m},$ where $\wedge$ denotes the minimum. Consider now an arbitrary $\lambda'\in(0,\infty)^\numbergroups$ with $\normtwo{\lambda-\lambda'} < \delta.$ 

As a next step,  note that Criterion~\eqref{eq:estimator} implies
\begin{equation*}
  g(\,\normtwo{\outcome-\design\estimatort}^2)+\sum_{j=1}^{\numbergroups}\lambda_j\,\normqj{\fusedmatrix_j\estimatort} \leq   g(\,\normtwo{\outcome}^2)\,.
\end{equation*}
In particular, since the function $g$ is non-negative on $[0,\infty)$ by assumption, it holds that
\begin{equation*}
 \sum_{j=1}^{\numbergroups}\lambda_j\,\normqj{\fusedmatrix_j\estimatort} \leq   g(\,\normtwo{\outcome}^2)\,.
\end{equation*}
Thus, we have 
\begin{equation*}
 \lambda_j\,\normqj{\fusedmatrix_j\estimatort} \leq   g(\,\normtwo{\outcome}^2) \text{~~for $j\in\{1,2,\dots,\numbergroups\}$} \,.
\end{equation*}
Hence, if $\lambda_j>0,$ it holds that
\begin{equation}
\label{eq:continuous1}
\normqj{\fusedmatrix_j\estimatort} \le \frac{g(\,\normtwo{\outcome}^2)}{\lambda_j }\text{~~for $j\in\{1,2,\dots,\numbergroups\}$} \,.
\end{equation}
Also note that by Criterion~\eqref{eq:estimator}, it holds that
\begin{align*}
g(\,\normtwo{\outcome-\design\estimatort}^2)+\sum_{j=1}^{\numbergroups}\lambda_j\,\normqj{\fusedmatrix_j\estimatort}&\le g(\,\normtwo{\outcome-\design\estimatortprime}^2)+\sum_{j=1}^{\numbergroups}\lambda_j\,\normqj{\fusedmatrix_j\estimatortprime}
\end{align*}
and
\begin{align*}
g(\,\normtwo{\outcome-\design\estimatortprime}^2)+\sum_{j=1}^{\numbergroups}\lambda'_j\,\normqj{\fusedmatrix_j\estimatortprime}&\le g(\,\normtwo{\outcome-\design\estimatort}^2)+\sum_{j=1}^{\numbergroups}\lambda'_j\,\normqj{\fusedmatrix_j\estimatort}\,.
\end{align*}
Rearranging these two inequalities, we obtain
\begin{align*}
g(\,\normtwo{\outcome-\design\estimatort}^2)+\sum_{j=1}^{\numbergroups}\lambda_j\,\normqj{\fusedmatrix_j\estimatort}-g(\,\normtwo{\outcome-\design\estimatortprime}^2)-\sum_{j=1}^{\numbergroups}\,\lambda'_j\,\normqj{\fusedmatrix_j\estimatortprime} &\le \sum_{j=1}^{\numbergroups}(\lambda_j-\lambda'_j)\,\normqj{\fusedmatrix_j\estimatortprime}\
\end{align*}
and
\begin{align*}
g(\,\normtwo{\outcome-\design\estimatort}^2)+\sum_{j=1}^{\numbergroups}\lambda_j\,\normqj{\fusedmatrix_j\estimatort}-g(\,\normtwo{\outcome-\design\estimatortprime}^2)-\sum_{j=1}^{\numbergroups}\,\lambda'_j\,\normqj{\fusedmatrix_j\estimatortprime} &\ge \sum_{j=1}^{\numbergroups}(\lambda_j-\lambda'_j)\,\normqj{\fusedmatrix_j\estimatort}\,.
\end{align*}
By H\"older's inequality and Inequality~\eqref{eq:continuous1}, it holds that
\begin{align*}
\sum_{j=1}^{\numbergroups}(\lambda_j-\lambda'_j)\,\normqj{\fusedmatrix_j\estimatort}&\ge - \,\normone{\lambda-\lambda'}\max_{1\le j\le \numbergroups}\big\{\,\normqj{\fusedmatrix_j\estimatort}\big\} \\
&\ge -\,\normone{\lambda-\lambda'}\max_{1\le j\le \numbergroups}\left\{\frac{g(\,\normtwo{\outcome}^2)}{\lambda_j}\right\}
\end{align*}
and similarly
\begin{align*}
\sum_{j=1}^{\numbergroups}(\lambda_j-\lambda'_j)\,\normqj{\fusedmatrix_j\estimatortprime}&\le \normone{\lambda-\lambda'}\max_{1\le j\le \numbergroups}\left\{\frac{g(\,\normtwo{\outcome}^2)}{\lambda'_j}\right\}.
\end{align*}
Now trivially, $\max_{1\le j\le \numbergroups}\left\{\frac{g(\,\normtwo{\outcome}^2)}{\lambda_j}\right\}\leq m$. Moreover,
\begin{align*}
\max_{1\le j\le \numbergroups}\left\{\frac{g(\,\normtwo{\outcome}^2)}{\lambda'_j}\right\}\leq \max_{1\le j\le \numbergroups}\left\{\frac{g(\,\normtwo{\outcome}^2)}{\lambda_j-\normtwo{\lambda-\lambda'}}\right\}\leq m
\end{align*}
by definition of $m$ and $\delta.$ Using this and again the definition of $\delta$, we then find
\begin{align*}
\Big | g(\,\normtwo{\outcome-\design\estimatort}^2)+\sum_{j=1}^{\numbergroups}\lambda_j\,\normqj{\fusedmatrix_j\estimatort}-g(\,\normtwo{\outcome-\design\estimatortprime}^2)-\sum_{j=1}^{\numbergroups}\lambda'_j\,\normqj{\fusedmatrix_j\estimatortprime} \Big|<\epsilon\,.
\end{align*}
This implies the desired continuity.

Now we show that function  $\lambda \mapsto g(\,\normtwo{\outcome-\design\estimatort}^2)$ is continuous.  We proceed with contradiction. Thus, we assume there exist $\lambda' \in (0,\infty)^\numbergroups$ and $\epsilon_0 >0,$ so that for any $\delta>0,$ there exists $\lambda \in (0,\infty)^\numbergroups$ satisfying
\begin{align*}
\normtwo{\lambda-\lambda'}<\delta \quad \text{and}\quad  \Big|\,g(\, \normtwo{\outcome-\design\estimatort}^2)-g(\, \normtwo{\outcome-\design\estimatortprime}^2)\Big|\ge \epsilon_0\,.
\end{align*}
We note that by Lemma~\ref{re:convexity}, the value of $g(\,\normtwo{\outcome-\design\estimatort}^2)$ does not depend on the specific  choice of the estimator~$\estimatort.$ 
Since $x\mapsto g(x)$ is strictly increasing, there exists $\epsilon'_0\equiv \epsilon'_0(\epsilon_0)>0$ such that 
\begin{equation}
\label{eq:continuous2}
\Big|\,\normtwo{\design\estimatort}^2-\normtwo{\design\estimatortprime}^2\Big|>\epsilon'_0\, .
\end{equation} 
Define the set $B:=\left\{\beta\in\R^p: \big|\,\normtwo{\design\estimatortprime}^2-\normtwo{\design\beta}^2\big|>\epsilon'_0\right\}$. It follows directly that $\estimatortprime\notin B$, and due to Inequality~\eqref{eq:continuous2} above, it follows that $\estimatort\in B.$ Let $\eta>0$ be arbitrary. Invoking Criterion~\eqref{eq:estimator} and the continuity of the objective function (note that taking minima does not affect the continuity), we obtain
\begin{align*}
&g(\,\normtwo{\outcome-\design\estimatort}^2)+\sum_{j=1}^\numbergroups\lambda_j\,\normqj{\fusedmatrix_j\estimatort}\\
=&\min_{\beta\in B}\big\{g(\,\normtwo{\outcome-\design\beta}^2)+\sum_{j=1}^\numbergroups
\lambda_j\,\normqj{\fusedmatrix_j\beta}\big\}\\
\geq &\min_{\beta\in B}\big\{g(\,\normtwo{\outcome-\design\beta}^2)+\sum_{j=1}^\numbergroups
\lambda'_j\,\normqj{\fusedmatrix_j\beta}\big\}-\eta
\end{align*}
if $\delta$ is sufficiently small. Moreover, using $\estimatortprime\notin B$ and again the continuity, it holds for $\delta$ sufficiently small that
\begin{align*}
& \min_{\beta\in B}\big\{g(\,\normtwo{\outcome-\design\beta}^2)+\sum_{j=1}^\numbergroups
\lambda'_j\,\normqj{\fusedmatrix_j\beta}\big\}\\
> & \min_{\beta\in \R^p}\big\{g(\,\normtwo{\outcome-\design\beta}^2)+\sum_{j=1}^\numbergroups
\lambda'_j\,\normqj{\fusedmatrix_j\beta}\big\}+\xi\\
\geq & \min_{\beta\in \R^p}\big\{g(\,\normtwo{\outcome-\design\beta}^2)+\sum_{j=1}^\numbergroups
\lambda_j\,\normqj{\fusedmatrix_j\beta}\big\}+\xi/2
\end{align*}
for a $\xi \equiv \xi (\lambda',\epsilon'_0)>0$. Choosing $\eta=\xi/4,$ we find
\begin{align*}
& \min_{\beta\in \R^p}\big\{g(\,\normtwo{\outcome-\design\beta}^2)+\sum_{j=1}^\numbergroups
\lambda_j\,\normqj{\fusedmatrix_j\beta}\big\}\\
> & \min_{\beta\in \R^p}\big\{g(\,\normtwo{\outcome-\design\beta}^2)+\sum_{j=1}^\numbergroups
\lambda_j\,\normqj{\fusedmatrix_j\beta}\big\}+\xi/4,
\end{align*}
which is a contradiction and thus concludes the proof of the continuity of the function $\tuningparameter\mapsto\linkfunction(\,\normtwo{\outcome-\design\estimatort}^2)$. The continuity of the function $\tuningparameter\mapsto\,\normtwo{\outcome-\design\estimatort}^2$ then follows from the assumption that the link function~\linkfunction\ is continuous and increasing. This concludes the proof of the lemma.
\end{proof}

\begin{proof}[Proof of Lemma~\ref{re:nonzero}]
Since $\outcome\neq \bold{0}_{n}$ with probability one, we assume $\outcome\neq \bold{0}_{n}$ in the following. 

We then show that 
\begin{equation*}
\normtwo{\outcome-\design\estimatort}>0\,.
\end{equation*}
We do this by contradiction, that is, we assume
\begin{equation*}
\normtwo{\outcome-\design\estimatort}=0\,.
\end{equation*}
This implies
\begin{equation}
\label{eq:assumption}
\outcome-\design\estimatort=\bold{0}_n\,.
\end{equation}
Since  $\beta\mapsto g(\,\normtwo{\outcome-\design\beta}^2)$ is convex, the subdifferential $\partial _\beta g(\,\normtwo{\outcome-\design\beta}^2)\Big|_{\beta=\estimatort}$ exists. Thus, the KKT conditions imply 
\begin{equation*}
\bold{0}_p\in\ \partial _\beta g(\,\normtwo{\outcome-\design\beta}^2)\Big|_{\beta=\estimatort}+\sum_{j=1}^\numbergroups \lambda_j\partial_\beta\,\normqj{\fusedmatrix_j\beta}\Big|_{\beta=\estimatort}\,,
\end{equation*}
which implies by the chain rule
\begin{equation*}
\bold{0}_p\in\partial_x g(x)\big|_{x=\,\normtwo{\outcome-\design \estimatort}^2}\big(-2\design^\top (\outcome-\design\estimatort)\big)+\sum_{j=1}^\numbergroups \lambda_j\partial_\beta\,\normqj{\fusedmatrix_j\beta}\big|_{\beta=\estimatort}\,.
\end{equation*}
Plugging Equality~\eqref{eq:assumption} into this display yields
\begin{equation*}
\bold{0}_p\in \sum_{j=1}^\numbergroups \lambda_j\partial_\beta\,\normqj{\fusedmatrix_j\beta}\big|_{\beta=\estimatort}\,.
\end{equation*}
This means that the vector~\estimatort\ minimizes the function $\beta\mapsto \sum_{j=1}^\numbergroups \lambda_j\,\normqj{\fusedmatrix_j\beta}.$ However, since $\tuningparameter_1,\dots,\tuningparameter_\numbergroups>0$, and by assumption on the matrices $\fusedmatrix_1,\dots,\fusedmatrix_\numbergroups,$ this mapping is a norm and is thus minimized only at ${\bf 0}_p.$ Consequently, $\estimatort={\bf 0}_p.$ However,  Equality~\eqref{eq:assumption} then gives 
\begin{equation*}
\outcome=\design\estimatort=\design\bold{0}_p=\bold{0}_n\,,
\end{equation*}
which contradicts $\outcome\neq\bold{0}_n.$ Thus, $\outcome-\design\estimatort\neq \bold{0}_n,$ and it follows that  $\normtwo{\outcome-\design\estimatort}^2\neq 0$. 

Since the function $x\mapsto g(x)$ is continuously differentiable on $(0,\infty)$ with strictly positive derivative, we finally obtain
\begin{equation*}
g'(\,\normtwo{\outcome-\design\estimatort}^2)> 0
\end{equation*}
as desired.
\end{proof}

\subsection{Proof of Lemma~\ref{re:existence}}
\begin{proof}[Proof of Lemma~\ref{re:existence}]
The proof consists of three steps.  First, we show that the solution equals zero if the tuning parameters are large enough. Second, we show that if one element of the tuning parameter is sufficiently large, increasing that element does not affect the estimator. Finally, we use these results to show the existence of suitable tuning parameters.

Let us start with some notation. For each $j\in\{1,\dots,\numbergroups\},$ we define the set  $A_j\subset \{1,2,\dots,p\}$ such that for any $u
\in A_j,$ the $u$th row of $\fusedmatrix_j$ is not zero, that is, $A_j:=\{u\in\{1,\dots,p\}:\max_{v\in\{1,\dots,p\}}|(\fusedmatrix_{j})_{uv}|>0\}.$ By assumption on the sequence  $\fusedmatrix_1,\dots,\fusedmatrix_\numbergroups$, it holds that $\bigcup_{j=1}^{\numbergroups}A_j=\{1,\dots,p\}.$ For $j\in\{1,\dots,k\},$ define the vector $(\design^\top \outcome)^j\in\R^p$  via $(\design^\top \outcome)^j_i:=(\design^\top \outcome)_i\1\{i\in A_j\}\,,$ and set $m:=\max_{j\in\{1,\dots,\numbergroups\}}2 g'(\,\normtwo{\outcome}^2)\,\normqjd{(\design^\top \outcome)^j}\vee 1,$ where $\vee$ denotes the maximum. 

The three mentioned steps now read in detail:
\begin{enumerate}
\item Show that for any $\lambda\in[m,\infty)^\numbergroups$, it holds that
\begin{align*}
\{\bold{0}_p\}=\argmin_{\beta\in\R^p}\big\{g(\,\normtwo{\outcome-\design\beta}^2)+\sum_{j=1}^\numbergroups \lambda_j\,\normqj{\fusedmatrix_j\beta}\big\}\,.
\end{align*}

\item Show that $\design\estimator^{\tilde\tuningparameter}=\design\estimatort$ if  $\tilde\lambda,  \lambda \in(0,\infty)^\numbergroups$  satisfy
\begin{align*}
\tilde\lambda_j
> \lambda_j =m &\text{~~~ if } j\in B \\ 
\tilde\lambda_j=\lambda_j & \text{~~~ if }  j\notin B
\end{align*}
for a non-empty subset $B \subset \{1,2,\dots,\numbergroups\}$.
\item  Show that with probability one, there exists a vector $\lambda\in(0,\infty)^\numbergroups$ that satisfies
 \begin{equation}\label{re:oracleexistence}
  \frac{\lambda_j}{2g'(\,\normtwo{\outcome-\design\estimatort}^2)}=\tuningrelax_j \,\normqjd{(\design \projectionmatrix_j \fusedmatrix_j^{+})^\top\noise}
 \end{equation}
 for all $j\in\{1,\dots,\numbergroups\}$.
\end{enumerate}

Step 1:
We first show that 
\begin{equation*}
\bold{0}_p\in \argmin_{\beta\in\R^p} \big\{ g(\,\normtwo{\outcome-\design\beta}^2)+\sum_{j=1}^\numbergroups\lambda_j\,\normqj{\fusedmatrix_j\beta}\big\}\,
\end{equation*}
implies
\begin{equation*}
\{\bold{0}_p\}=\argmin_{\beta\in\R^p} \big\{ g(\,\normtwo{\outcome-\design\beta}^2)+\sum_{j=1}^\numbergroups\lambda_j\,\normqj{\fusedmatrix_j\beta}\big\}\,.
\end{equation*}
Assume for a $\lambda\in[m,\infty)^\numbergroups$, it holds that
\begin{equation*}
\bold{0}_p\in \argmin_{\beta\in\R^p} \big\{ g(\,\normtwo{\outcome-\design\beta}^2)+\sum_{j=1}^\numbergroups\lambda_j\,\normqj{\fusedmatrix_j\beta}\big\}\,.
\end{equation*}
then by Lemma~\ref{re:convexity}, for any $\estimatort\in\argmin_{\beta\in\R^p} \big\{ g(\,\normtwo{\outcome-\design\beta}^2)+\sum_{j=1}^\numbergroups\lambda_j\,\normqj{\fusedmatrix_j\beta}\big\}\,,$ it holds that $\design \estimatort=\bold{0}_n\,.$
Hence, by assumption on the matrices $\fusedmatrix_1,\dots,\fusedmatrix_\numbergroups,$ it holds for any $\estimatort\neq {\bf 0}_p,$
\begin{align*}
&g(\,\normtwo{\outcome-\design\estimatort}^2)+\sum_{j=1}^\numbergroups\lambda_j\,\normqj{\fusedmatrix_j\estimatort}\\
=&g(\,\normtwo{\outcome}^2)+\sum_{j=1}^\numbergroups\lambda_j\,\normqj{\fusedmatrix_j\estimatort}\\
> &g(\,\normtwo{\outcome}^2)+\sum_{j=1}^\numbergroups\lambda_j\,\normqj{\fusedmatrix_j{\bf 0}_p}\\
=&g(\,\normtwo{\outcome-\design{\bf 0}_p}^2)+\sum_{j=1}^\numbergroups\lambda_j\,\normqj{\fusedmatrix_j{\bf 0}_p}\,.
\end{align*}
This contradicts $\estimatort\in\argmin_{\beta\in\R^p} \big\{ g(\,\normtwo{\outcome-\design\beta}^2)+\sum_{j=1}^\numbergroups\lambda_j\,\normqj{\fusedmatrix_j\beta}\big\}\,,$ and thus, $\estimatort={\bf 0}_p.$  Thus,
\begin{equation*}
\{\bold{0}_p\}=\argmin_{\beta\in\R^p} \big\{ g(\,\normtwo{\outcome-\design\beta}^2)+\sum_{j=1}^\numbergroups\lambda_j\,\normqj{\fusedmatrix_j\beta}\big\}\,
\end{equation*}
as desired.
It is left to show that for any vector $\lambda\in[m,\infty)^\numbergroups$, it holds that
\begin{equation*}
\bold{0}_p\in \argmin_{\beta\in\R^p} \big\{ g(\,\normtwo{\outcome-\design\beta}^2)+\sum_{j=1}^\numbergroups\lambda_j\,\normqj{\fusedmatrix_j\beta}\big\}\,.
\end{equation*}
By the KKT conditions, we have to show that there are vectors $\kappa(j)\in\partial_{\beta}\,\normqj{\fusedmatrix_j\beta}\Big|_{\beta=\bold{0}_p}$ such that 
\begin{equation}
\label{eq:subdiff1}
-2g'(\,\normtwo{\outcome}^2)\design^\top \outcome+\sum_{j=1}^\numbergroups\lambda_j \kappa(j)=\bold{0}_p\,.
\end{equation}
 Define $\tilde A_1:=A_1, \tilde A_j:=A_j\setminus \{A_1,\dots,A_{j-1}\} \text{ for } j=2,\dots,\numbergroups$. In particular, $\tilde A_j\subset A_j$, the $\tilde A_j$'s are disjoint, and $\bigcup_{j=1}^\numbergroups \tilde A_j=\{1,\dots,p\}$. 
With this notation, we need to show that for any $v\in\tilde A_j$ and $j\in\{1,\dots,\numbergroups\}$, it holds that 
\begin{align*}
-2g'(\,\normtwo{\outcome}^2)(\design^\top \outcome)_v+\sum_{j=1}^\numbergroups\lambda_j \kappa(j)_v=0\,.
\end{align*}
Define the vector $\kappa(j)$ for $j\in\{1,\dots,\numbergroups\}$ via
\begin{align*}
\kappa(j)_v:=\begin{cases}
\frac{2g'(\,\normtwo{\outcome}^2)(\design^\top \outcome)^j_v}{\lambda_j} &\text{ if } v\in \tilde A_j \\ 
0 & \text{ if } v\notin \tilde A_j \,.
\end{cases}
\end{align*}
Since $\tilde A_j\subset A_j$ and $\tilde A_j$'s are disjoint, we find for all $j\in\{1,\dots,\numbergroups\}$ and $v\in \tilde A_j,$ that
\begin{align*}
-2g'(\,\normtwo{\outcome}^2)(\design^\top \outcome)_v+\sum_{j=1}^\numbergroups\lambda_j \kappa(j)_v=-2g'(\,\normtwo{\outcome}^2)(\design^\top \outcome)^j_v+\lambda_j \kappa(j)_v\,.
\end{align*}
By definition of the vectors $\kappa(j)$ for $j\in\{1,\dots,\numbergroups\},$ we thus have
\begin{align*}
-2g'(\,\normtwo{\outcome}^2)(\design^\top \outcome)^j_v+\lambda_j \kappa(j)_v=0\,,
\end{align*}
which implies
\begin{align*}
-2g'(\,\normtwo{\outcome}^2)(\design^\top \outcome)_v+\sum_{j=1}^\numbergroups\lambda_j \kappa(j)_v=0\,.
\end{align*}
Since  $\lambda_j\ge m\ge 2g'(\,\normtwo{\outcome}^2)\,\normqjd{ (\design^\top \outcome)^j}$, it also follows by taking the dual norm on both sides that
\begin{equation}\label{eq:boundedlambda}
\normqjd{\kappa(j)}\le \frac{2g'(\,\normtwo{\outcome}^2)\,\normqjd{(\design^\top \outcome)^j}}{\lambda_j}\leq 1\,,
\end{equation}
and hence, $\kappa(j)\in\partial_{\beta}\, \normqj{\fusedmatrix_j\beta}\Big|_{\beta=\bold{0}_p}$ for all $j\in\{1,\dots,\numbergroups\}$\,. So we get
\begin{equation*}
\bold{0}_p\in\argmin_{\beta\in\R^p} \big\{ g(\,\normtwo{\outcome-\design\beta}^2)+\sum_{j=1}^\numbergroups\lambda_j\,\normqj{\fusedmatrix_j\beta}\big\}\,.
\end{equation*}
as desired.
We conclude that 
\begin{equation*}
\{\bold{0}_p\}=\argmin_{\beta\in\R^p} \big\{ g(\,\normtwo{\outcome-\design\beta}^2)+\sum_{j=1}^\numbergroups\lambda_j\,\normqj{\fusedmatrix_j\beta}\big\}\,.
\end{equation*}

Step 2: 
Consider a pair of vectors $\lambda, \tilde \lambda\in (0,\infty)^\numbergroups,$ that satisfy $\tilde \lambda_j> \lambda_j=m$ for $j\in B$ and $\tilde\lambda_j=\lambda_j$ for $j\in \{1,\dots,\numbergroups\}\setminus B.$
For $\lambda$, fix a solution
\begin{equation*}
\estimatort\in \argmin_{\beta\in\R_p} \big\{ g(\,\normtwo{\outcome-\design\beta})+\sum_{j=1}^\numbergroups\lambda_j\,\normqj{\fusedmatrix_j\beta}\big\}
\end{equation*}
with corresponding subdifferentials $\kappa(j)\in\partial_{\beta} \,\normqj{\fusedmatrix_j\beta}\Big|_{\beta=\estimatort}$ that satisfy the KKT conditions
\begin{align*}
-2g'(\,\normtwo{\outcome-\design\estimatort}^2)\design^\top (\outcome-\design\estimatort)+\sum_{j=1}^\numbergroups \lambda_j \kappa(j)=\bold{0}_p\,.
\end{align*}
We first need to show that 
\begin{equation*}
\estimatort\in \argmin_{\beta\in\R_p} \big\{ g(\,\normtwo{\outcome-\design\beta})+\sum_{j=1}^\numbergroups\tilde\lambda_j\,\normqj{\fusedmatrix_j\beta}\big\}\,.
\end{equation*}
By the KKT conditions, we have to show that there are the vectors $\tilde\kappa(j)\in\partial_{\beta}\, \normqj{\fusedmatrix_j\beta}\Big|_{\beta=\estimatort}$ such that
\begin{align*}
-2g'(\,\normtwo{\outcome-\design\estimatort}^2)\design^\top (\outcome-\design\estimatort)+\sum_{j=1}^\numbergroups\tilde \lambda_j \tilde\kappa(j)=\bold{0}_p\,.
\end{align*}
Define $\tilde \kappa(j)$ for $j\in\{1,\dots,\numbergroups\}$ via
\begin{align*}
\tilde\kappa(j):=\frac{\lambda_j}{\tilde \lambda_j}\kappa(j) \,.
\end{align*}
Plugging this into the previous display yields
\begin{align*}
-2g'(\,\normtwo{\outcome-\design\estimatort}^2)\design^\top (\outcome-\design\estimatort)+\sum_{j=1}^\numbergroups\tilde \lambda_j \tilde \kappa(j)=-2g'(\,\normtwo{\outcome-\design\estimatort}^2)\design^\top( \outcome-\design\estimatort)+\sum_{j=1}^\numbergroups\lambda_j \kappa_j\,.
\end{align*}
Therefore,  it holds that
\begin{align*}
-2g'(\,\normtwo{\outcome-\design\estimatort}^2)\design^\top (\outcome-\design\estimatort)+\sum_{j=1}^\numbergroups\tilde \lambda_j \tilde \kappa(j)=\bold{0}_p\,.
\end{align*}
Moreover, by definition of $\tilde \kappa(j)$ and Inequality~\eqref{eq:boundedlambda}, we have that 
\begin{equation*}
\normqjd{\tilde\kappa(j)}=\frac{\lambda_j}{\tilde \lambda_j}\,\normqjd{\kappa(j)} \leq \normqjd{\kappa(j)}\le 1\,,
\end{equation*}
for all $j\in B$. So,  $\tilde \kappa(j)\in\partial_{\beta}\, \normqj{\fusedmatrix_j\beta}\Big|_{\beta=\estimatort}$ for any $j\in\{1,\dots,\numbergroups\}$.
Hence, it holds that
\begin{equation*}
\estimatort\in \argmin_{\beta\in\R^p} \big\{ g(\,\normtwo{\outcome-\design\beta})+\sum_{j=1}^\numbergroups\tilde\lambda_j\,\normqj{(\fusedmatrix_j\beta)^j}\big\}\,.
\end{equation*}
This gives $\design\estimator^{\tilde\tuningparameter}=\design\estimatort$  by  Lemma~\ref{re:convexity}.\vspace{2mm}
  
Step 3: Finally, we show the existence of a $\lambda\in (0,\infty)^\numbergroups$ that satisfies Equality~\eqref{re:oracleexistence}. For this, we define
\begin{equation*}
  a:=m\wedge \left(2\linkfunction'(\,\normtwo{\outcome}^2)\min_{j\in\{1,\dots,\numbergroups\}} \tuningrelax_j\,\norm{(\design \projectionmatrix_j\fusedmatrix_j^{+})^\top\noise}_{q_j}^*\right)\,.
\end{equation*}
By assumption on the noise~$\noise,$ it holds that $a>0$ with probability one. Next, we consider the function 
\begin{align*}
f: [a,m]^\numbergroups&\to \R^\numbergroups\\
\lambda&\mapsto f(\tuningparameter):=2g'(\,\normtwo{\outcome-\design\estimator^{\lambda}}^2)\left(\tuningrelax_1\,\norm{(\design \projectionmatrix_1\fusedmatrix_1^{+})^\top\noise}_{q_1}^*,\dots,\tuningrelax_k\,\norm{(\design \projectionmatrix_k\fusedmatrix_k^{+})^\top\noise}_{q_k}^*\right)^\top.
\end{align*}
Note first that $\linkfunction(\,\normtwo{\outcome-\design\estimatort}^2)\leq \linkfunction(\,\normtwo{\outcome}^2)$ by definition of the estimator~\estimatort.  Hence, since  $\linkfunction'$ is non-increasing, we find
\begin{align*}
 \min_{\lambda\in[a,m]^\numbergroups}  \min_{j\in\{1,\dots,\numbergroups\}} f(\lambda)_j\geq a\,.
\end{align*}
Note also that $[a,m]^\numbergroups$ is compact and that $f$ is continuous by Lemmas~\ref{re:functions} and~\ref{re:nonzero} and the assumption that $g'(x)$  is continuous on $(0,\infty)$. It therefore holds that 
\begin{equation*}
\sup_{\lambda\in[a,m]^\numbergroups} \normsup{f(\lambda)} \leq b
\end{equation*}
for some  $b\in(0,\infty).$ Using this and Step 2, we find that the function $f$ and the function
\begin{align*}
h: [a,b\vee m]^\numbergroups&\to [a,b\vee m]^\numbergroups\\
\lambda&\mapsto h(\tuningparameter):=2g'(\,\normtwo{\outcome-\design\estimator^{\lambda}}^2)\left(\tuningrelax_1\,\norm{(\design \projectionmatrix_1\fusedmatrix_1^{+})^\top\noise}_{q_1}^*,\dots,\tuningrelax_k\,\norm{(\design \projectionmatrix_k\fusedmatrix^{+}_k)^\top\noise}_{q_k}^*\right)^\top
\end{align*}
have the same images, that is,
\begin{align*}
\left\{y: y=f(\lambda),\lambda\in[a,m]^\numbergroups \right\}=\left\{y: y=h(\lambda),\lambda\in[a,b\vee m]^\numbergroups \right\}\,.
\end{align*}
The function  $h$ is continuous. Moreover, $[a,b\vee m]^\numbergroups$ is a compact and convex subset of~$\R^k$. We can thus apply  Brouwer's fixed-point theorem to deduce that 
\begin{align*}
\lambda=2g'(\,\normtwo{\outcome-\design\estimator^{\lambda}}^2)\left(\tuningrelax_1\,\norm{(\design \projectionmatrix_1\fusedmatrix_1^{+})^\top\noise}_{q_1}^*,\dots,\tuningrelax_k\,\norm{(\design \projectionmatrix_k\fusedmatrix_k^{+})^\top\noise}_{q_k}^*\right)^\top
\end{align*}
for a vector  $\lambda\in [a,b\vee m]^\numbergroups$. According to Lemma~\ref{re:nonzero}, it holds  that $g'(\,\normtwo{\outcome-\design\estimatort}^2)>0$ with probability one, so that
 \begin{equation*}
   \frac{\tuningparameter}{2g'(\,\normtwo{\outcome-\design\estimator^{\lambda}}^2)}=\left(\tuningrelax_1\,\norm{(\design  \projectionmatrix_1\fusedmatrix_1^{+})^\top \noise}_{q_1}^*,\dots, \tuningrelax_k\,\norm{(\design  \projectionmatrix_\numbergroups\fusedmatrix_\numbergroups^{+})^\top \noise}_{q_\numbergroups}^*\right)^\top
  \end{equation*}
as desired.
\end{proof}

\subsection{Proof of Theorem~\ref{r:finalbound}}


\label{page:argument}
\begin{proof}[Proof of Theorem~\ref{r:finalbound}]
Consider the function
\begin{align*}
\R^p&\to \R\\
\beta&\mapsto f(\beta):=g(\,\normtwo{\outcome-\design\beta}^2)+\sum_{j=1}^{\numbergroups}\lambda_j\,\normqj{\fusedmatrix_j\beta}\,.
\end{align*}
According to our assumptions on the link function and the penalties, $f$ is convex, and $\estimatort$ minimizes $f.$  
Therefore, ${\bf 0}_p\in\partial f(\beta)|_{\beta=\estimatort}.$ 
Subdifferentials are additive, so that we can write $\partial f(\beta)|_{\beta=\estimatort}$ as a sum of subdifferentials of the individual parts of $f.$ In particular, we can decompose ${\bf 0}_p\in\partial f(\beta)|_{\beta=\estimatort}$ as 
\begin{equation*}
  {\bf 0}_p= \frac{\partial}{\partial \beta}g(\,\normtwo{\outcome-\design\beta}^2)\Big|_{\beta=\estimatort}+\sum_{j=1}^\numbergroups\tuningparameter_j\kappa_j\,,
\end{equation*}
where $\kappa_j\in\partial\{\alpha\mapsto \normqj{M_{j}\alpha}\}|_{\alpha=\estimatort}$ and, using the assumption that the function $g$ is differentiable on $(0,\infty),$
\begin{equation*}
 \frac{\partial}{\partial \beta}g(\,\normtwo{\outcome-\design\beta}^2)\Big|_{\beta=\estimatort}=-2 g'(\,\normtwo{\outcome-\design\estimatort}^2)\big(\design^{\top}(\outcome-\design\estimatort)\big)\,.
\end{equation*}
Adding the pieces together implies for all $\beta\in\R^p$
\begin{equation*}
  0={\bf 0}_p^\top(\genbeta-\estimatort)=\big(-2 g'(\,\normtwo{\outcome-\design\estimatort}^2)\big(\design^{\top}(\outcome-\design\estimatort)\big)+\sum_{j=1}^\numbergroups\tuningparameter_j\kappa_j\big)^\top(\genbeta-\estimatort)\,.
\end{equation*}
Now, $\kappa_j\in\partial\{\alpha\mapsto \normqj{M_{j}\alpha}\}|_{\alpha=\estimatort},$ which means by the definition of subdifferentials for convex functions that  for all $\genbeta\in \R^p$
\begin{equation*}
\normqj{\fusedmatrix_{j}\genbeta}\ge \normqj{\fusedmatrix_{j}\estimatort}+\inprod{\kappa_j}{\genbeta-\estimatort}\,,
\end{equation*}
which is equivalent to
\begin{align*}
\normqj{\fusedmatrix_j\genbeta}-\normqj{\fusedmatrix_j\estimatort}\ge \kappa_j^\top (\genbeta-\estimatort)\,.
\end{align*}
Combining with the above yields
\begin{align*}
&-2g'(\,\normtwo{\outcome-\design\estimatort}^2) \big(\design^\top(\outcome-\design\estimatort)\big)^\top(\genbeta-\estimatort)
+\sum_{j=1}^{\numbergroups}\lambda_j\big(\,\normqj{ \fusedmatrix_j \genbeta}-\normqj{ \fusedmatrix_j\estimatort}\big)\,\geq\,0 \,. 
\end{align*}
According to the model~\eqref{eq:model}, we can replace $\outcome$ with  $\design\betatrue+\noise$ to obtain
\begin{align*}
&\big(\design^\top(\outcome-\design\estimatort)\big)^\top(\genbeta-\estimatort)\\
 =&\big(\design^\top(\design\betatrue+\noise-\design\estimatort)\big)^\top(\genbeta-\estimatort)\\
 = &(\design(\betatrue-\estimatort))^\top \design(\genbeta-\estimatort)+\noise^\top \design(\genbeta-\estimatort)\\
= &(\design(\betatrue-\estimatort))^\top \design(\betatrue-\estimatort+\genbeta-\betatrue)+\noise^\top \design(\genbeta-\estimatort)\\
 =& \,\normtwo{\design(\betatrue-\estimatort)}^2+ (\design(\betatrue-\estimatort))^\top \design(\genbeta-\betatrue)+ \noise^\top \design(\genbeta-\estimatort)\,.
\end{align*}
Now, we note that for any $\genconstant>0,$
  \begin{equation*}   
 (\design(\betatrue-\estimatort))^\top \design(\genbeta-\betatrue)
\geq -\genconstant\,\normtwo{\design(\betatrue-\estimatort)}^2-\frac{\normtwo{\design(\genbeta-\betatrue)}^2}{4\genconstant}\,.
  \end{equation*}
Using this in the foregoing display and consolidating gives us
  \begin{equation*}
    (\design^\top(\outcome-\design\estimatort))^\top(\genbeta-\estimatort)\geq \,(1-\genconstant)\,\normtwo{\design(\betatrue-\estimatort)}^2-\frac{\normtwo{\design(\genbeta-\betatrue)}^2}{4\genconstant} + \noise^\top \design(\genbeta-\estimatort)\,.
  \end{equation*}
Plugging this back into the earlier display yields, noting that $\linkfunction'$ is a positive function by assumption,
\begin{multline*}
 -2g'(\,\normtwo{\outcome-\design\estimatort}^2)\big((1-\genconstant)\, \normtwo{\design(\betatrue-\estimatort)}^2-\frac{\normtwo{\design(\genbeta-\betatrue)}^2}{4\genconstant} +\langle \noise, \design(\genbeta-\estimatort)\rangle\big)\\
+\sum_{j=1}^{\numbergroups}\lambda_j\big(\,\normqj{ \fusedmatrix_j \genbeta}-\normqj{\fusedmatrix_j\estimatort}\big)\ge 0\,.
\end{multline*}
Rearranging this inequality, we obtain
\begin{multline*}
2g'(\,\normtwo{\outcome-\design\estimatort}^2)\big((1-\genconstant)\, \normtwo{\design(\betatrue-\estimatort)}^2+\langle \noise, \design(\genbeta-\estimatort)\rangle \big)\\ \le \sum_{j=1}^{\numbergroups}\lambda_j\big(\,\normqj{ \fusedmatrix_j \genbeta}-\normqj{\fusedmatrix_j\estimatort}\big)+2g'(\,\normtwo{\outcome-\design\estimatort}^2)\,\frac{\normtwo{\design(\genbeta-\betatrue)}^2}{4\genconstant}\,.
\end{multline*}
According to Lemma~\ref{re:nonzero},  we can divide both sides by $2g'(\,\normtwo{\outcome-\design\estimatort}^2)$, so that
\begin{multline*}
 (1-\genconstant)\,\normtwo{\design(\betatrue-\estimatort)}^2+\langle \noise, \design(\genbeta-\estimatort)\rangle \\
\le \sum_{j=1}^{\numbergroups}\frac{\lambda_j}{2g'(\,\normtwo{\outcome-\design\estimatort}^2)}\big(\,\normqj{ \fusedmatrix_j \genbeta}-\normqj{ \fusedmatrix_j\estimatort}\big)+\frac{\normtwo{\design(\genbeta-\betatrue)}^2}{4\genconstant}
\end{multline*}
and
\begin{multline*}
(1-\genconstant)\, \normtwo{\design(\betatrue-\estimatort)}^2 \\
\le \langle \noise, \design(\estimatort-\genbeta)\rangle+ \sum_{j=1}^{\numbergroups}\frac{\lambda_j}{2g'(\,\normtwo{\outcome-\design\estimatort}^2)}(\,\normqj{ \fusedmatrix_j \genbeta}-\,\normqj{\fusedmatrix_j\estimatort})+\frac{\normtwo{\design(\genbeta-\betatrue)}^2}{4\genconstant}
\end{multline*}
with probability one.
Recall that by Equation~\eqref{eq:partition},  the vector $\estimatort-\genbeta$ can be rewritten as
\begin{align*}
\estimatort-\genbeta=\sum_{j=1}^\numbergroups \projectionmatrix_{j} \fusedmatrix_j^+\fusedmatrix_{j} (\estimatort-\genbeta)\,.
\end{align*}
So we can reorganize the inner product $\langle \noise, \design(\estimatort-\genbeta) \rangle$ according to
\begin{align*}
 \langle \noise, \design(\estimatort-\genbeta)\rangle
&= \langle \noise, \design \sum_{j=1}^\numbergroups \projectionmatrix_{j} \fusedmatrix_j^+\fusedmatrix_{j} (\estimatort-\genbeta) \rangle\\
&= \sum_{j=1}^\numbergroups \langle \noise, \design\projectionmatrix_{j} \fusedmatrix_j^+\fusedmatrix_{j} (\estimatort-\genbeta)\rangle\\
&=\sum_{j=1}^\numbergroups  \langle  (\design  \projectionmatrix_j\fusedmatrix_{j}^+)^\top\noise, \fusedmatrix_{j} (\estimatort-\genbeta)\rangle\,.
\end{align*}
Plugging this into the previous inequality yields
\begin{align*}
&(1-\genconstant)\,\normtwo{\design(\betatrue-\estimatort)}^2 \\      
\le & \sum_{j=1}^\numbergroups\big( \langle  (\design   \projectionmatrix_j \fusedmatrix_j^{+})^\top\noise, \fusedmatrix_j (\estimatort-\genbeta)\rangle+\frac{\lambda_j}{2g'(\,\normtwo{\outcome-\design\estimatort}^2)}(\,\normqj{\fusedmatrix_j \genbeta}-\,\normqj{\fusedmatrix_j\estimatort})\big)+\frac{\normtwo{\design(\genbeta-\betatrue)}^2}{4\genconstant}\,.
\end{align*}
Using H\"older's Inequality, we can rewrite the first term in the second line according to
\begin{align*}
 \langle  (\design  \projectionmatrix_j \fusedmatrix_j^{+})^\top\noise, \fusedmatrix_j (\estimatort-\genbeta)\rangle \le \normqjd{(\design   \projectionmatrix_j \fusedmatrix_j^{+})^\top\noise}\,\normqj{\fusedmatrix_j\estimatort-\fusedmatrix_j \genbeta}\,.
\end{align*}
Plugging this back into the previous display gives
\begin{multline*}
(1-\genconstant)\,\normtwo{\design(\betatrue-\estimatort)}^2 
\leq  \sum_{j=1}^\numbergroups\big(\,\normqjd{(\design  \projectionmatrix_j \fusedmatrix_j^{+})^\top\noise}\,\normqj{\fusedmatrix_j\estimatort-\fusedmatrix_j \genbeta}\\+\frac{\lambda_j}{2g'(\,\normtwo{\outcome-\design\estimatort}^2)}(\,\normqj{\fusedmatrix_j \genbeta}-\,\normqj{\fusedmatrix_j\estimatort})\big)
+\frac{\normtwo{\design(\genbeta-\betatrue)}^2}{4\genconstant}\,.
\end{multline*}
We can modify this further by applying the triangle inequality and by reorganizing the terms. We find
\begin{align*}
&(1-\genconstant)\,\normtwo{\design(\betatrue-\estimatort)}^2 \\ 
\le & \sum_{j=1}^\numbergroups\big(\,\normqjd{(\design  \projectionmatrix_j \fusedmatrix_j^{+})^\top\noise}(\,\normqj{\fusedmatrix_j\estimatort}+\,\normqj{\fusedmatrix_j \genbeta})+\frac{\lambda_j}{2g'(\,\normtwo{\outcome-\design\estimatort}^2)}(\,\normqj{\fusedmatrix_j \genbeta}-\,\normqj{\fusedmatrix_j\estimatort})\big)\\
&+\frac{\normtwo{\design(\genbeta-\betatrue)}^2}{4\genconstant}\\
=& \sum_{j=1}^\numbergroups\big(\,\normqjd{(\design  \projectionmatrix_j \fusedmatrix_j^{+})^\top\noise}+\frac{\lambda_j}{2g'(\,\normtwo{\outcome-\design\estimatort}^2)}\big)\,\normqj{\fusedmatrix_j \genbeta}\\
& +\sum_{j=1}^\numbergroups\big(\,\normqjd{(\design \projectionmatrix_j \fusedmatrix_j^{+})^\top\noise}-\frac{\lambda_j}{2g'(\,\normtwo{\outcome-\design\estimatort}^2)}\big)\,\normqj{\fusedmatrix_j\estimatort}+\frac{\normtwo{\design(\genbeta-\betatrue)}^2}{4\genconstant}\,.
\end{align*}
We now set $\lambda$ according to Lemma~\ref{re:existence}. It then holds that\begin{align*}
\frac{\lambda_j}{2g'(\,\normtwo{\outcome-\design\estimatort}^2)}= \tuningrelax_j\,\normqjd{(\design  \projectionmatrix_j\fusedmatrix_j^{+})^\top\noise}
\end{align*}
 for all $j\in\{1,\dots,\numbergroups\}.$
 Then it follows that
 \begin{multline*}
(1-\genconstant)\,\normtwo{\design(\betatrue-\estimatort)}^2\\
 \le \sum_{j=1}^{\numbergroups}(1+\tuningrelax_j)\, \normqjd{(\design  \projectionmatrix_j\fusedmatrix_j^{+})^\top\noise} \,  \normqj{\fusedmatrix_j \genbeta}+\sum_{j=1}^{\numbergroups}(1-\tuningrelax_j)\, \normqjd{(\design  \projectionmatrix_j\fusedmatrix_j^{+})^\top\noise} \,  \normqj{\fusedmatrix_j \estimatort}+\frac{\normtwo{\design(\betatrue-\genbeta)}^2}{4\genconstant}\,.
\end{multline*}
To bring this into the standard form, assuming $\genconstant<1,$ we finally divide both sides by $(1-\genconstant)n$ and find
\begin{multline*}
\frac{1}{n}\,\normtwo{\design(\betatrue-\estimatort)}^2\le \frac{1}{n}\sum_{j=1}^{\numbergroups}\frac{1+\tuningrelax_j}{1-\genconstant}\,  \normqjd{(\design  \projectionmatrix_j\fusedmatrix_j^{+})^\top\noise}  \, \normqj{\fusedmatrix_j \genbeta}\\+\frac{1}{n}\sum_{j=1}^{\numbergroups}\frac{1-\tuningrelax_j}{1-\genconstant}\, \normqjd{(\design  \projectionmatrix_j\fusedmatrix_j^{+})^\top\noise} \,  \normqj{\fusedmatrix_j \estimatort}+\frac{1}{4\genconstant(1-\genconstant)n}\,\normtwo{\design(\betatrue-\genbeta)}^2\,.
\end{multline*}
Since $\genbeta\in\R^p$ and $\genconstant\in(0,1)$ were arbitrary, this inequality provides us with the desired statement.\label{page:last}

\end{proof}

\end{document}